\documentclass[[a4paper,UKenglish,cleveref, autoref, thm-restate]{lipics-v2021}
\usepackage{xcolor}
\usepackage{extarrows}
\usepackage{amsthm}
\usepackage{amsmath}
\usepackage{amssymb}
\usepackage{amsfonts}
\usepackage{graphicx}
\usepackage[ruled]{algorithm2e}
\usepackage{mathtools}

\allowdisplaybreaks
\nolinenumbers

\newtheorem{thrm}{Theorem}[section]
\newtheorem{lem}[thrm]{Lemma}
\newtheorem{prop}[thrm]{Proposition}

\theoremstyle{definition}
\newtheorem{defn}[thrm]{Definition}
\theoremstyle{definition}

\theoremstyle{definition}

\theoremstyle{definition}

\theoremstyle{definition}
\newtheorem{fct}[thrm]{Fact}

\newcounter{ProblemCounter}

\newcommand{\sgmG}{\langle \mathcal{G} \rangle}
\newcommand{\Z}{\mathbb{Z}}
\newcommand{\N}{\mathbb{N}}
\newcommand{\Q}{\mathbb{Q}}
\newcommand{\Rpp}{\mathbb{R}_{> 0}}
\newcommand{\Rp}{\mathbb{R}_{\geq 0}}

\newcommand{\A}{\mathbb{A}}

\newcommand{\R}{\mathbb{R}}

\newcommand{\Zpp}{\mathbb{Z}_{>0}}

\newcommand{\mG}{\mathcal{G}}

\newcommand{\hG}{\widehat{\mathcal{G}}}
\newcommand{\mS}{\mathcal{S}}
\newcommand{\mC}{\mathcal{C}}
\newcommand{\mM}{\mathcal{M}}
\newcommand{\mH}{\mathcal{H}}

\newcommand{\bg}{\boldsymbol{g}}
\newcommand{\balpha}{\boldsymbol{\alpha}}

\newcommand{\sgn}{\operatorname{sgn}}

\newcommand{\RXs}{\mathbb{R}_{\geq 0}[X^{\pm}]^*}
\newcommand{\NXds}{\mathbb{N}[X^{\pm d}]^*}
\newcommand{\RX}{\mathbb{R}[X^{\pm}]}
\newcommand{\bff}{\boldsymbol{f}}
\newcommand{\lc}{\operatorname{lc}}
\newcommand{\hw}{\widehat{w}}
\newcommand{\hy}{\widehat{y}}
\newcommand{\hb}{\widehat{b}}

\newcommand{\init}{\operatorname{in}}
\newcommand{\coef}{\operatorname{coef}}

\newcommand{\SL}{\mathsf{SL}}

\title{The Identity Problem in $\mathbb{Z} \wr \mathbb{Z}$ is decidable}
\author{Ruiwen Dong}{Department of Computer Science, University of Oxford}{ruiwen.dong@kellogg.ox.ac.uk}{}{}

\authorrunning{R. Dong}

\Copyright{Ruiwen Dong} 



\category{} 

\relatedversion{}

\supplement{}

\acknowledgements{The author would like to thank Markus Schweighofer and David Sawall for useful discussions. The author acknowledges support from UKRI Frontier Research Grant EP/X033813/1.}

\begin{document}

\maketitle
\begin{abstract}
We consider semigroup algorithmic problems in the wreath product $\mathbb{Z} \wr \mathbb{Z}$.
Our paper focuses on two decision problems introduced by Choffrut and Karhum\"{a}ki (2005): the \emph{Identity Problem} (does a semigroup contain the neutral element?) and the \emph{Group Problem} (is a semigroup a group?) for finitely generated sub-semigroups of $\mathbb{Z} \wr \mathbb{Z}$.
We show that both problems are decidable.
Our result complements the undecidability of the \emph{Semigroup Membership Problem} (does a semigroup contain a given element?) in $\mathbb{Z} \wr \mathbb{Z}$ shown by Lohrey, Steinberg and Zetzsche (ICALP 2013), and contributes an important step towards solving semigroup algorithmic problems in general metabelian groups.
\end{abstract}

\section{Introduction}
The computational theory of groups and semigroups is one of the
oldest and most well-developed parts of computational algebra.
Dating back to the work of Markov~\cite{markov1947certain} in the 1940s, the area plays an essential role in analysing system dynamics, with notable applications in automata theory and program analysis~\cite{blondel2005decidable, choffrut2005some, derksen2005quantum, hrushovski2018polynomial}.
See~\cite{kharlampovich1995algorithmic} for an all-encompassing survey on this topic.
Among the most prominent problems in this area are \emph{Semigroup Membership} and \emph{Group Membership}, proposed respectively by Markov and Mikhailova in the 1940s and 1960s.
For these decision problems, we work in a fixed group $G$. The input is a finite set of elements $\mG \subseteq G$, plus a distinguished element $g \in G$. 
Denote by $\sgmG$ the semigroup generated by $\mG$, and by $\sgmG_{grp}$ the group generated by $\mG$.
\begin{enumerate}[(i)]
    \item \textit{(Semigroup Membership)} decide whether $\sgmG$ contains $g$.
    \item \textit{(Group Membership)} decide whether $\langle\mG\rangle_{grp}$ contains $g$.
    \setcounter{ProblemCounter}{\value{enumi}}
\end{enumerate}
In this paper, we consider two problems closely related to (i) and (ii): the \emph{Identity Problem} and the \emph{Group Problem}, both introduced by Choffrut and Karhum\"{a}ki~\cite{choffrut2005some} in 2005.
\begin{enumerate}[(i)]
    \setcounter{enumi}{\value{ProblemCounter}}
    \item \textit{(Identity Problem)} decide whether $\sgmG$ contains the neutral element $I$ of $G$.
    \item \textit{(Group Problem)} decide whether $\langle\mG\rangle$ is a group, in other words, whether $\sgmG = \langle\mG\rangle_{grp}$.
    \setcounter{ProblemCounter}{\value{enumi}}
\end{enumerate}
In general matrix groups, Semigroup Membership is undecidable by a classical result of Markov~\cite{markov1947certain}.
All four problems remain undecidable even for integer matrix groups of dimension four~\cite{bell2010undecidability, mikhailova1966occurrence}.
Notably, using an embedding of the \emph{Identity Correspondence Problem}, Bell and Potapov~\cite{bell2010undecidability} showed undecidability of the Identity Problem and the Group Problem in the group $\SL(4, \Z)$ of $4 \times 4$ integer matrices of determinant one.
On the other hand, in the group $\SL(2, \Z)$, all four problems are decidable with various degrees of complexity~\cite{bell2017identity, choffrut2005some, lohrey2021subgroup}.
In particular, the Identity Problem and the Group Problem in $\SL(2, \Z)$ are both \textbf{NP}-complete by a result of Bell, Hirvensalo, and Potapov \cite{bell2017identity}.

In this paper we focus on these decision problems in the \emph{wreath product} $\Z \wr \Z$.
The wreath product is a fundamental construction in group and semigroup theory.
A great number of important groups can be constructed using the wreath product, notably \emph{metabelian groups}. 
Metabelian groups are groups whose commutator is abelian: these are the simplest generalization of abelian groups.
Algorithmic problems in metabelian groups have been the focus of active research since the 1970s~\cite{baumslag1994algorithmic, chapuis1997free, roman1979equations}, with a classic result of Romanovskii~\cite{romanovskii1974some} showing decidability of Group Membership in all finitely presented metabelian groups.
A key part of Romanovskii's proof is to embed metabelian groups into quotients of wreath products.
In fact, the Magnus embedding theorem~\cite{magnus1939theorem} states that every finitely generated free metabelian group can be embedded in a wreath product $\Z^m \wr \Z^n$.
Therefore, understanding the wreath product $\Z \wr \Z$ is the most crucial step towards studying general metabelian groups.
Apart from its interest within group theory, the wreath product also plays an important role in the algebraic theory of automata.
The Krohn–Rhodes theorem~\cite{krohn1965algebraic} states that every finite semigroup (and correspondingly, every finite automaton) can be decomposed into elementary components using wreath products.

One easy way to understand the wreath product $\Z \wr \Z$ is through its isomorphism to a matrix group~\cite{magnus1939theorem} over the Laurent polynomial ring $\Z[X^{\pm}]$:
\begin{align}\label{eq:defphi}
    \Z \wr \Z \cong \left\{ 
\begin{pmatrix}
        X^{b} & y \\
        0 & 1
\end{pmatrix}
\;\middle|\; y \in \Z[X^{\pm}], b \in \Z 
\right\}.
\end{align}
Consider the four aforementioned decision problems in $\Z \wr \Z$.
Since $\Z \wr \Z$ is metabelian~\cite{kharlampovich2020diophantine}, the classic result of Romanovskii~\cite{romanovskii1974some} shows decidability of Group Membership in $\Z \wr \Z$.
If we retrace the proof of Romanovskii, one can reduce Group Membership in $\Z \wr \Z$ to solving systems of linear equations over the ring $\Z[X^{\pm}]$, which can then be decided using Gr\"{o}bner basis.
For Semigroup Membership in $\Z \wr \Z$, Lohrey, Steinberg and Zetzsche showed its undecidability using an encoding of 2-counter machines~\cite{lohrey2015rational}.
Decidability of the Identity Problem and the Group Problem in $\Z \wr \Z$ remained an intricate open problem.
A recent paper by Dong~\cite{dong2022solving} gave a partial decidability result when the generators all satisfy $b = \pm 1$.
Dong's idea was to represent a product of elements in $\Z \wr \Z$ as a walk on $\Z$.
When the generators all satisfy $b = \pm 1$, this walk can be decomposed into simple cycles, and the Group Problem reduces to solving a \emph{single} homogeneous linear equation over the semiring $\N[X]$.
Extending Dong's result to arbitrary generators is highly challenging: the structure of the walk becomes much more complex when we allow steps of arbitrary length.
In this paper, we combine a series of new ideas from graph theory and algebraic geometry to show full decidability of the Identity Problem and the Group Problem in $\Z \wr \Z$.

The first main idea of this paper is to reduce both problems to solving a \emph{system} of homogeneous linear equations over the semiring $\N[X^{\pm}]$, in addition to two \emph{degree constraints}.
We use a highly non-trivial graph theoretic construction to establish this reduction.
The second main idea of this paper is to generalize a local-global principle by Einsiedler, Mouat and Tuncel~\cite{einsiedler2003does} to solve these linear equations with degree constraints.
In particular, the original local-global principle by Einsiedler et al.\ is not compatible with the degree constraints which are essential in our reduction.
We introduce new ideas to prove a generalized local-global principle that incorporates these additional degree constraints.

We now mention some other known decidability results in wreath products.
In~\cite{ganardi2017knapsack}, Ganardi, K{\"o}nig, Lohrey and Zetzsche showed that for every non-trivial finitely generated abelian group $G$, the \emph{knapsack problem} in $G \wr \Z$ is \textbf{NP}-complete. Notably, this result applies to $\Z \wr \Z$.
In~\cite{lohrey2015rational}, Lohrey, Steinberg and Zetzsche showed decidability of the \emph{Rational Subset Membership Problem} (which subsumes all four decision problems mentioned in the beginning) in the wreath product $H \wr V$, where $H$ is a finite and $V$ is virtually free.
In~\cite{DBLP:conf/icalp/CadilhacCZ20}, Cadilhac, Chistikov and Zetzsche proved decidability of Rational Subset Membership in the \emph{Baumslag-Solitar groups} $\mathsf{BS}(1, p)$.
This group can be considered as an analogue of $\left(\Z / p \Z\right) \wr \Z$ ``with carrying''.
In~\cite{kharlampovich2020diophantine}, Kharlampovich, L{\'o}pez, and Myasnikov showed decidability of solving Diophantine equations in certain metabelian groups, including $\Z \wr \Z$ and $\mathsf{BS}(1, p)$.
Many of these results are closely related to automata theory, which we draw inspiration from.

A natural follow-up to our work would be trying to solve the Identity Problem and the Group Problem in \emph{all} finitely presented metabelian groups. This boils down to deciding both problems in quotients of $\Z^m \wr \Z^n$.
One encounters some difficulties when trying to generalize our approach to $\Z^m \wr \Z^n$.
Notably, decomposition of walks in $\Z^n$ are much more complex, and we can no longer reduce these problems to solving a finite system of equations.
We also point out that one cannot go much further beyond metabelian groups (which are 2-step solvable groups), since there exist 3-step solvable groups with undecidable word problem~\cite{kharlampovich1981finitely}.

\section{Preliminaries}
\subsection*{Words, semigroups and graphs}
Let $G$ be an arbitrary group.
Let $\mG = \{g_1, \ldots, g_a\}$ be a finite set of elements in $G$.
Considering $\mG$ as an alphabet, denote by $\mG^*$ the set of words over $\mG$.
For an arbitrary word $w = g_{i_1} g_{i_2} \cdots g_{i_m} \in \mG^*$, by multiplying consecutively the elements appearing in $w$, we can evaluate $w$ as an element $\pi(w)$ in $G$.
We say that the word $w$ \emph{represents} the element $\pi(w)$.
The semigroup $\langle \mG \rangle$ generated by $\mG$ is hence the set of elements in $G$ that are represented by \emph{non-empty} words in $\mG^*$.

A word $w$ over the alphabet $\mG$ is called \emph{full-image} if every letter in $\mG$ has at least one occurrence in $w$.
The following observation shows that deciding the Group Problem amounts to finding a full-image word representing the neutral element.

\begin{restatable}{lem}{lemgrpword}\label{lem:grpword}
    Let $\mG = \{g_1, \ldots, g_a\}$ be a set of elements in a group $G$.
    The semigroup $\langle \mG \rangle$ is a group if and only if the neutral element $I$ of $G$ is represented by a full-image word over $\mG$.
\end{restatable}

The following lemma shows that decidability of the Group Problem implies decidability of the Identity Problem.

\begin{lem}[\cite{bell2010undecidability}]\label{lem:grptoid}
    Given a finite subset $\mG$ of a group $G$, the semigroup $\langle \mG \rangle$ contains the neutral element $I$ if and only if there exists a non-empty subset $\mH \subseteq \mG$ such that $\langle \mH \rangle$ is a group.
    In particular, if the Group Problem is decidable in the group $G$, then the Identity Problem is also decidable.
\end{lem}

For detailed definition of graph theory terms, see~\cite{bondy1976graph}.
All graphs considered in this paper will be directed multigraphs.
For a graph $G$, we denote by $V(G)$ its set of vertices and by $E(G)$ its set of edges.
For a (directed) edge $e$, we denote by $s(e)$ the starting vertex of $e$.

A \emph{loop} is an edge that starts and ends at the same vertex.
A \emph{circuit} is a path that starts and ends at the same vertex.
An \emph{Euler path} of a graph $G$ is a path that uses each edge exactly once.
An \emph{Euler circuit} is an Euler path that starts and ends at the same vertex.
We call a graph \emph{Eulerian} if it contains an Euler circuit.
It is easy to see that attaching a circuit to an Eulerian graph still results in an Eulerian graph.

\subsection*{Laurent polynomials and the wreath product $\Z \wr \Z$}
A (univariate) \emph{Laurent polynomial} with coefficients over $\R$ is an expression of the form
\[
f = \sum_{i = p}^q a_i X^i, \quad \text{where $p, q \in \Z$ and $a_i \in \R, i = p, p+1, \ldots, q$.}
\]
If $p > q$, then $f$ is understood to be zero. 
Otherwise, $p \leq q$, and we suppose $a_p \neq 0$, $a_q \neq 0$.
In this case, we call $p$ the \emph{negative degree} of $f$, denoted by $\deg_-(f)$, and call $q$ the \emph{positive degree} of $f$, denoted by $\deg_+(f)$.
We call $a_p$ the \emph{negative leading coefficient} of $f$, denoted by $\lc_-(f)$, and call $a_q$ the \emph{positive leading coefficient} of $f$, denoted by $\lc_+(f)$.
Define additionally $\deg_-(0) = + \infty$, $\deg_+(0) = - \infty$ and $\lc_-(0) = \lc_+(0) = 0$.
\underline{In this paper, all polynomials will be univariate Laurent polynomials.}
The set of all polynomials with coefficients over $\R$ forms a ring and is denoted by $\R[X^{\pm}]$.
One can define $\Q[X^{\pm}]$ and $\Z[X^{\pm}]$ similarly by restricting the coefficients $a_i$ to $\Q$ and $\Z$.

Given a tuple of  polynomials $\bff = (f_1, \ldots, f_n) \in \left(\R[X^{\pm}]\right)^n$ and $r \in \R$, one naturally defines the evaluation $\bff(r) \coloneqq (f_1(r), \ldots, f_n(r)) \in \R^n$.
The definition of leading coefficients also extends to tuples of polynomials by 
$
\lc_*(\bff) \coloneqq (\lc_*(f_1), \ldots, \lc_*(f_n)) \in \R^n
$, where $* \in \{+, -\}$.

Consider the semiring $\Rp[X^{\pm}]$ of  polynomials with positive coefficients: these are expressions of the form $f = \sum_{i = p}^q a_i X^i$ where $p, q \in \Z$ and $a_i \in \Rp, i = p, p+1, \ldots, q$.
Define further $\Rp[X^{\pm}]^* \coloneqq \Rp[X^{\pm}] \setminus \{0\}$.
One can define $\N[X^{\pm}]$ and $\N[X^{\pm}]^*$ similarly by restricting the coefficients $a_i$ to $\N$.

An element $f = \sum_{i = p}^q a_i X^i \in \Rp[X^{\pm}]^*$ is called \emph{gap-free} if $a_i \neq 0$ for all $i = p, p+1, \ldots, q$.
Is it easy to see that, given arbitrary $M, N \in \Zpp$ and $f \in \Rp[X^{\pm}]^*$, the polynomial $(X^{-M} + X^{-M + 1} +  \cdots + X^N)^n \cdot f$ is gap-free for all large enough $n$.

A \emph{monomial} is a  polynomial $f = a_i X^i$ with only one term (including zero).
Let $d \geq 1$ be a positive integer. One can define the semirings 
\[
\N[X^{\pm d}] \coloneqq \left\{\sum_{i = p}^q a_{di} X^{di} \in \N[X^{\pm}] \right\}, \quad \N[X^{\pm d}]^* \coloneqq \N[X^{\pm d}] \setminus \{0\}.
\]
These are  polynomials whose monomials have degrees divisible by $d$.
Similarly, if $a_{di} \neq 0$ for all $i = p, \ldots, q$, we will call $\sum_{i = p}^q a_{di} X^{di}$ gap-free.
Note that whether a polynomial is gap-free depends on the polynomial ring we consider it in.

Similarly one can define the rings $\Z[X^{\pm d}]$, $\Q[X^{\pm d}]$ and $\R[X^{\pm d}]$.
Furthermore, we define the field of rational functions $\Q(X)$ to be the set of expressions of the form $\frac{f}{g}$, where $f, g \in \Q[X^{\pm}]$.
Similarly, $\Q(X^{d})$ is defined as the set of expressions $\frac{f}{g}$, where $f, g \in \Q[X^{\pm d}]$.

The wreath product $\Z \wr \Z$ has several equivalent definitions.
Here, we introduce the one most convenient to our purpose.

\begin{defn}\label{def:wr}
    The wreath product $\Z \wr \Z$ is a group whose elements are pairs of the form $(y, b)$, where $y \in \Z[X^{\pm}]$ and $b \in \Z$.
    The neutral element in $\Z \wr \Z$ is given by $(0, 0)$.
    Multiplication is defined by
    $
    (y, b) \cdot (y', b') = (y + X^b \cdot y', b + b')
    $,
    and inversion is defined by $(y, b)^{-1} = (X^{-b} \cdot y, - b)$.   
    Note that the element $(y, b)$ corresponds to the matrix 
    $\begin{pmatrix}
        X^{b} & y \\
        0 & 1
    \end{pmatrix}$
    under the isomorphism~\eqref{eq:defphi} in the introduction.
\end{defn}

The wreath product $\Z \wr \Z$ can be embedded into the larger group $\Q(X) \rtimes \Z$, whose elements are pairs of the form $(y, b)$ with $y \in \Q(X)$ and $b \in \Z$ and whose multiplication and inversion are defined using the same formulas as in $\Z \wr \Z$.

\section{Overview of proof}\label{sec:overview}
The main result of this paper is the decidability of the Identity Problem and the Group Problem in $\Z \wr \Z$.
In view of Lemma~\ref{lem:grptoid}, it suffices to prove decidability of the Group Problem.
In this section we give an overview of its proof.

Our proof proceeds in three steps.
As a first step we reduce the Group Problem in $\Z \wr \Z$ to deciding whether a system of linear equations in $\R[X^{\pm}]$ has solution in $\Rp[X^{\pm}]^*$ with two additional degree constraints (see Proposition~\ref{prop:wrtoeq} and Corollary~\ref{cor:sys}).
As the second step we prove a local-global principle that further reduces solving linear equations over $\Rp[X^{\pm}]^*$ to solving a family of ``local'' equations over $\Rpp$ (see Proposition~\ref{prop:locglob}).
As the third step, we show that solving these ``local'' equations can be done using the first order theory of the reals as well as Gr\"{o}bner basis techniques (see Proposition~\ref{prop:declcd} and \ref{prop:declcinf}).

Let $\mG$ be a finite subset of $\Z \wr \Z$.
Write $\mG = \{(y_a, b_a) \mid a \in A\}$ where $A$ is a finite set of indices.
Divide $A$ into three subsets of indices $A = I \cup J \cup K$ where 
\begin{equation}\label{eq:defijk}
    I \coloneqq \{i \mid b_i > 0\}, \quad J \coloneqq \{j \mid b_j < 0\}, \quad K \coloneqq \{k \mid b_k = 0\}.
\end{equation}

First, we exclude the easy case where $I$ or $J$ is empty.
\begin{restatable}{prop}{propeasy}\label{prop:easy}
    Suppose $I = \emptyset$ or $J = \emptyset$.
    The semigroup $\langle \mG \rangle$ is a group if and only if $I = J = \emptyset$ and $\sum_{k \in K} n_k y_k = 0$ for some positive integers $n_k \in \Zpp$.
    In particular, this is decidable by integer programming.
\end{restatable}

A simple proof of Proposition~\ref{prop:easy} is given in the Appendix~\ref{app:proofs}.
\underline{For the rest of this paper,} \underline{we will suppose $I \neq \emptyset$ and $J \neq \emptyset$.}

Define
\[
d \coloneqq \gcd\left( \{b_a \mid  a \in I \cup J\} \right).
\]
For each pair $(i, j) \in I \times J$, define the rational function
\begin{equation}\label{eq:defhij}
    h_{(i,j)} \coloneqq \frac{y_i}{1 + X^d + \cdots + X^{b_i - d}} + \frac{y_j}{X^{-d} + X^{-2d} + \cdots + X^{- |b_j|}} \in \Q(X).
\end{equation}
By direct computation we have
\begin{equation}\label{eq:effectele}
(y_i, b_i)^{|b_j|} \cdot (y_j, b_j)^{b_i} = \left(h_{(i,j)} \cdot (1 + X^d + \cdots + X^{b_i |b_j| - d}), 0\right).
\end{equation}
One can also take \eqref{eq:effectele} as the definition of $h_{(i,j)}$.

A subset $S$ of $I \times J$ is called \emph{double-full} if for every $i \in I$ there exists $j_i \in J$ such that $(i, j_i) \in S$, and for every $j \in J$ there exists $i_j \in I$ such that $(i_j, j) \in S$.
The following proposition reduces the Group Problem in $\Z \wr \Z$ to solving linear equations over $\NXds$.

\begin{restatable}{prop}{propwrtoeq}\label{prop:wrtoeq}
The semigroup $\langle \mG \rangle$ is a group if and only if there exist a double-full set $S \subset I \times J$ and polynomials $f_{(i,j)}, f_k \in \NXds$ for $(i, j) \in S, k \in K$ that satisfy the following three conditions.
\begin{enumerate}[(i)]
    \item \textbf{(Single linear equation)} The following equation over $\Q(X)$ is satisfied:
    \begin{equation}\label{eq:Xd}
        \sum_{(i, j) \in S} f_{(i,j)} \cdot h_{(i,j)} + \sum_{k \in K} f_k \cdot y_{k} = 0.
    \end{equation}
    \item \textbf{(Positive degree bound)} We have:
    \begin{equation}\label{eq:deg+}
        \deg_{+}\left(\sum\nolimits_{(i, j) \in S} f_{(i,j)}\right) + d \geq \deg_{+}\left(\sum\nolimits_{k \in K} f_{k}\right).
    \end{equation}
    \item \textbf{(Negative degree bound)} We have:
    \begin{equation}\label{eq:deg-}
        \deg_{-}\left(\sum\nolimits_{(i, j) \in S} f_{(i,j)}\right) \leq \deg_{-}\left(\sum\nolimits_{k \in K} f_{k}\right).
    \end{equation}
\end{enumerate}
\end{restatable}

The proof of Proposition~\ref{prop:wrtoeq} will be given in Section~\ref{sec:wrtopoly}.
The idea is roughly as follows.
By Lemma~\ref{lem:grpword}, $\langle \mG \rangle$ is a group if and only if there is a full-image word $w \in \mG^*$ that represents the neutral element.
For the ``only if'' statement of Proposition~\ref{prop:wrtoeq}, we will represent $w$ as a walk over $\Z$, and decompose the walk into ``primitive circuits''.
Each primitive circuit contributes a multiple of $h_{(i,j)}$ or $y_k$ to the element represented by $w$.
Since $w$ represents the neutral element, this results in the linear equation in Condition~(i). Conditions~(ii) and (iii) will stem from the connectedness of the walk.
The ``if'' statement is significantly harder. Given the polynomials $f_{(i,j)}, f_k$, we will construct an Eulerian graph $G$ by attaching long ``elementary circuits'' in a way that corresponds to the coefficients of $f_{(i,j)}, f_k$.
We then read a word $w$ from an Euler circuit of $G$.
Condition~(i) will make sure $w$ represents the neutral element.
However, making sure $G$ is connected is highly non-trivial and will be the main difficulty of the proof.
In particular, Conditions~(ii)-(iii) will be crucial.
This concludes the idea of the proof for Proposition~\ref{prop:wrtoeq}.

Note that Proposition~\ref{prop:wrtoeq} involves finding solutions over $\NXds$ for linear equations with coefficients in $\Q(X)$.
When $d > 1$, this is inconvenient, so we now further reduce Proposition~\ref{prop:wrtoeq} to finding solutions over $\Rp[X^{\pm}]^*$ for a \emph{system} of linear equations.
For each $(i, j) \in I \times J$, since the denominator of $h_{(i,j)}$ is an element in $\Q[X^{\pm d}]$, there exist $h_{(i,j), 0}, \ldots, h_{(i,j), d-1} \in \Q(X)$ such that $h_{(i,j)}$ can be written as
\begin{equation}\label{eq:hsys}
h_{(i,j)} = h_{(i,j),0}(X^d) + h_{(i,j),1}(X^d) \cdot X + \cdots + h_{(i,j),d-1}(X^d) \cdot X^{d-1}.
\end{equation}
Similarly, for each $k \in K$, there exist $y_{k, 0}, \ldots, y_{k, d-1} \in \Q[X^{\pm}]$ so that $y_{k}$ can be written as
\begin{equation}\label{eq:ysys}
y_{k} = y_{k,0}(X^d) + y_{k,1}(X^d) \cdot X + \cdots + y_{k,d-1}(X^d) \cdot X^{d-1}.
\end{equation}
The following corollary shows that, using the elements $h_{(i,j), m}, y_{k, m}$ defined in~\eqref{eq:hsys} and \eqref{eq:ysys}, we can rewrite the conditions in Proposition~\ref{prop:wrtoeq} using only variables in $\Rp[X^{\pm}]^*$ instead of $\NXds$.
See Appendix~\ref{app:proofs} for a simple proof of Corollary~\ref{cor:sys}.
\begin{restatable}{cor}{corsys}\label{cor:sys}
The semigroup $\langle \mG \rangle$ is a group if and only if there exist a double-full set $S \subset I \times J$ and polynomials $f_S, f_K, f_{(i,j)}, f_k \in \RXs$ for $(i, j) \in S, k \in K$ that satisfy the following three conditions.
\begin{enumerate}[(i)]
    \item \textbf{(System of linear equations)} The following linear equations over $\R(X)$ are satisfied:
    \begin{align}
        & \sum_{(i, j) \in S} f_{(i,j)} h_{(i,j),m} + \sum_{k \in K} f_k y_{k, m} = 0, \quad m = 0, \ldots, d-1, \label{eq:syscor1} \\
        & f_S = \sum_{(i, j) \in S} f_{(i,j)}, \quad\quad
        f_K = \sum_{k \in K} f_k. \label{eq:syscor3}
    \end{align}
    \item \textbf{(Positive degree bound)} We have:
    \begin{equation}\label{eq:deg+cor}
        \deg_{+}\left(f_S\right) + 1 \geq \deg_{+}\left(f_K\right).
    \end{equation}
    \item \textbf{(Negative degree bound)} We have:
    \begin{equation}\label{eq:deg-cor}
        \deg_{-}\left(f_S\right) \leq \deg_{-}\left(f_K\right).
    \end{equation}
\end{enumerate}
\end{restatable}

For brevity, from now on we denote $\A \coloneqq \RX$ and $\A^+ \coloneqq \RXs$.
Denote also $n \coloneqq 2 + |S| + |K|$.
Define the following subset of $\A^{n}$:
\begin{equation}\label{eq:defM}
    \mM \coloneqq \left\{ \bff = \left(f_S, f_K, (f_{(i,j)})_{(i, j) \in S}, (f_{k})_{k \in K} \right) \in \A^n \;\middle|\; \bff \text{ satisfies \eqref{eq:syscor1} and \eqref{eq:syscor3}} \right\}.
\end{equation}
That is, $\mM$ is the set of solutions of the linear equations~\eqref{eq:syscor1}-\eqref{eq:syscor3}.
Using linear algebra over the polynomial ring $\A$ (see~\cite{bareiss1968sylvester}), one can effectively compute a set of vectors $\bg_1, \ldots, \bg_m \in \A^{n}$ such that
\begin{equation}\label{eq:baseM}
\mM = \left\{ \phi_1 \bg_1 + \cdots + \phi_m \bg_m \;\middle|\; \phi_1, \ldots, \phi_m \in \A \right\}.
\end{equation}
One can even suppose $\bg_1, \ldots, \bg_m \in \left(\Q[X^{\pm}]\right)^{n}$ since $h_{(i,j)}$ and $y_k$ all have rational coefficients.
A set $\mM$ of the form~\eqref{eq:baseM} will be called a \emph{$\A$-submodule} of $\A^{n}$, and the elements $\bg_1, \ldots, \bg_m$ will be called a \emph{basis} of $\mM$.

Corollary~\ref{cor:sys} actually states the following: the semigroup $\langle \mG \rangle$ is a group if and only if $\mM$ contains an element $\bff = (f_S, f_K, \cdots) \in \left(\A^+\right)^{n}$ such that $\deg_+(f_S) + 1 \geq \deg_+(f_K)$ and $\deg_-(f_S) \leq \deg_-(f_K)$.
The key to deciding the existence of $\bff$ is the following proposition, which can be considered as a local-global principle that generalizes a result of Einsiedler, Mouat and Tuncel~\cite{einsiedler2003does}.
\begin{restatable}{prop}{proplocglob}\label{prop:locglob}
Let $\mM$ be a $\A$-submodule of $\A^{n}$.
Then $\mM$ contains an element $\bff = (f_S, f_K, \cdots) \in \left(\A^+\right)^{n}$ with $\deg_+(f_S) + 1 \geq \deg_+(f_K)$ and $\deg_-(f_S) \leq \deg_-(f_K)$, if and only if the following three conditions are all satisfied.
\begin{enumerate}[(i)]
    \item \textbf{(Existence of $\bff_r$ for all $r \in \R_{>0}$)} For each $r \in \R_{>0}$, there exists $\bff_r \in \mM$ such that 
    \begin{equation}\label{eq:locr}
        \bff_r(r) \in \R_{>0}^{n}.
    \end{equation}
    \item \textbf{(Existence of $\bff_{\infty}$)} There exists $\bff_{\infty} = (f_{\infty,S}, f_{\infty,K}, \cdots) \in \mM$ such that 
    \begin{equation}\label{eq:loc+}
        \lc_+(\bff_{\infty}) \in \R_{>0}^{n} \quad \text{ and } \quad \deg_{+}\left(f_{\infty,S}\right) + 1 \geq \deg_{+}\left(f_{\infty,K}\right).
    \end{equation}
    \item \textbf{(Existence of $\bff_{0}$)} There exists $\bff_{0} = (f_{0, S}, f_{0, K}, \cdots) \in \mM$ such that 
    \begin{equation}\label{eq:loc-}
        \lc_-(\bff_{0}) \in \R_{>0}^{n} \quad \text{ and } \quad \deg_{-}\left(f_{0, S}\right) \leq \deg_{-}\left(f_{0, K}\right).
    \end{equation}
\end{enumerate}
\end{restatable}
The proof of Proposition~\ref{prop:locglob} will be given in Section~\ref{sec:locglob}.
The original result of Einsiedler et al.~\cite[Theorem~1.3]{einsiedler2003does} gives a similar local-global principle without the degree constraints on $f_S$ and $f_K$.
While our proof follows the main steps of the original proof, we need to introduce new arguments in order for the degree constraint to stay compatible with the local-global principle.

One direction of the implication in Proposition~\ref{prop:locglob} is clear.
In fact, if $f \in \A^+$ and $r \in \Rpp$, then we have $f(r) \in \Rpp$ and $\lc_{+}(f), \lc_{-}(f) \in \Rpp$.
Therefore, if $\mM$ contains an element $\bff = (f_S, f_K, \cdots) \in \left(\A^+\right)^{n}$ with $\deg_+(f_S) + 1 \geq \deg_+(f_K)$ and $\deg_-(f_S) \leq \deg_-(f_K)$, then simply take $\bff_r = \bff_{\infty} = \bff_{0} = \bff$ for all $r$:
Equation~\eqref{eq:locr} is satisfied for all $r$ as well as \eqref{eq:loc+} and \eqref{eq:loc-}; hence all three conditions are satisfied.

On the other hand, if the Equation~\eqref{eq:locr} as well as \eqref{eq:loc+} and \eqref{eq:loc-} can be satisfied \emph{individually} by different $\bff_r, \bff_{\infty}, \bff_{0}$, we cannot \emph{a priori} find an element $\bff \in \mM$ in $\left(\A^+\right)^{n}$.
Such an element $\bff$ would \emph{simultaneously} satisfy Equations~\eqref{eq:locr} for all $r \in \R_{>0}$ as well as \eqref{eq:loc+} and \eqref{eq:loc-}.
The key idea of proving this non-trivial direction is that if all three conditions (i)-(iii) are satisfied, then we can ``glue'' these different $\bff_r, \bff_{\infty}$ and $\bff_{0}$ together to obtain a single $\bff$ that satisfies Equation~\eqref{eq:locr} for \emph{all} $r$ as well as \eqref{eq:loc+} and \eqref{eq:loc-}.
While this idea comes from the original proof, the difficult part in our generalization is to make sure the degree constraints are still satisfied after the gluing procedure.
In the end, we multiply this $\bff$ by a ``large enough'' polynomial to obtain an element in $\left(\A^+\right)^{n}$, using a theorem of Handelman (Theorem~\ref{cor:Handelman}).

The following two propositions show that Conditions~(i), (ii) and (iii) of Proposition~\ref{prop:locglob} are all decidable.

\begin{restatable}{prop}{propdeclcd}\label{prop:declcd}
Let $\mM$ be an $\A$-submodule of $\A^{n}$.
Given as input a finite basis of $\mM$, it is decidable whether for every $r \in \R_{>0}$ there exists $\bff_r \in \mM$ with $\bff_r(r) \in \R_{>0}^{n}$.
\end{restatable}

\begin{restatable}{prop}{propdeclcinf}\label{prop:declcinf}
Let $* \in \{+, -\}$, $a \in \Z$ and $\mM$ be an $\A$-submodule of $\A^{n}$.
Given as input a finite basis of $\mM$, it is decidable whether there exists $\bff = (f_S, f_K, \cdots) \in \mM$ such that 
    \begin{equation}\label{eq:declcinf}
        \lc_*(\bff) \in \R_{>0}^{n} \quad \text{ and } \quad \deg_{*}\left(f_S\right) + a \geq \deg_{*}\left(f_K\right).
    \end{equation}
\end{restatable}
Proposition~\ref{prop:declcd} and \ref{prop:declcinf} will be proven in Section~\ref{sec:locdec}.
The idea of proving Proposition~\ref{prop:declcd} is to reduce the statement to the first order theory of the reals; while for proving Proposition~\ref{prop:declcinf} we will use the \emph{super Gr\"{o}bner basis} introduced in the original proof of Einsiedler et al.~\cite{einsiedler2003does}.

We are now ready to prove our main theorem by bridging the remaining gaps.

\begin{restatable}{thrm}{thmid}\label{thm:id}
The Identity Problem and the Group Problem in $\Z \wr \Z$ are decidable.
\end{restatable}
\begin{proof}
    First we show decidability of the Group Problem.
    A detailed description of the full procedure is given in Appendix~\ref{app:alg} as a reference point for the readers.
    Given a finite set $\mG = \{(y_a, b_a) \mid a \in A\}$ in $\Z \wr \Z$, define the index sets $I, J, K$ as in \eqref{eq:defijk}.
    If $I$ or $J$ is empty, then Proposition~\ref{prop:easy} shows that the Group Problem for $\mG$ is decidable.
    If $I$ and $J$ are not empty, we enumerate all double-full sets $S \subset I \times J$.
    For each $S$ we compute a finite basis of $\mM$ defined in \eqref{eq:defM}.
    Corollary~\ref{cor:sys} together with Proposition~\ref{prop:locglob} shows that the Group Problem for $\mG$ has a positive answer if and only if for some $S$, the three conditions in Proposition~\ref{prop:locglob} are all satisfied.
    For each $S$, Condition~(i) can be decided using Proposition~\ref{prop:declcd}; Condition~(ii) can be decided using Proposition~\ref{prop:declcinf} by taking $* = +, a = 1$; Condition~(iii) can be decided using Proposition~\ref{prop:declcinf} by swapping the coordinates $S$ and $K$ and taking $* = -, a = 0$.
    Therefore the Group Problem in $\Z \wr \Z$ is decidable.
    By Lemma~\ref{lem:grptoid}, the Identity Problem in $\Z \wr \Z$ is also decidable.
\end{proof}

\section{From semigroup to polynomial equations}\label{sec:wrtopoly}

\subsection{Definition of $\mG$-graphs}
Section~\ref{sec:wrtopoly} is dedicated to the proof of Proposition~\ref{prop:wrtoeq}.
In this subsection, we will define the notion of a \emph{$\mG$-graph}.
Let $A$ be a finite set of indices.
Let $\mG = \{(y_a, b_a) \mid a \in A\}$ be a finite set of elements in the group $\Z \wr \Z$ or $\Q(X) \rtimes \Z$.
We define the following notion of a $\mG$-graph.

\begin{defn}[$\mG$-graphs]
A \emph{$\mG$-graph} is a directed multigraph $G$, whose set of vertices $V(G)$ is a finite subset of $\Z$, and its edges are each labeled with an index in $A$.
Furthermore, if an edge from vertex $d_1$ to vertex $d_2$ has label $a$, then $d_2 = d_1 + b_a$.
\end{defn}

For a word $w$ over the alphabet $\mG$, we associate to it a unique $\mG$-graph $G(w)$, defined as follows.
Write $w = (y_{a_1}, b_{a_1}) \cdots (y_{a_p}, b_{a_p})$. For each $i = 0, \ldots, p-1$, we add an edge starting at the vertex $b_{a_1} + \cdots + b_{a_{i}}$, ending at the vertex $b_{a_1} + \cdots + b_{a_{i+1}}$, with the label $a_i$. (If $i = 0$ then the edge starts at $0$ and ends at $b_{a_1}$.)
The graph $G(w)$ is then obtained by taking the connected component of the vertex $0$.
See Figure~\ref{fig:Gw} for the illustration of an example.

\begin{figure}[ht]
    \centering
    \begin{minipage}[t]{.45\textwidth}
        \centering
        \includegraphics[width=0.9\textwidth,height=0.85\textheight,keepaspectratio, trim={6.8cm 2.5cm 5.0cm 2.5cm},clip]{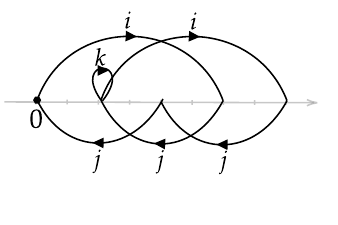}
        \caption{Illustration of $G(w)$. Here, \\ $A = \{i, j, k\}$, $\mG = \{(y_i, 6), (y_j, -4), (y_k, 0)\}$, \\ $w = (y_i, 6)(y_j, -4)(y_k, 0)(y_i, 6)(y_j, -4)(y_j, -4)$.}
        \label{fig:Gw}
    \end{minipage}
    \hfill
    \begin{minipage}[t]{0.45\textwidth}
        \centering
        \includegraphics[width=0.9\textwidth,height=0.85\textheight,keepaspectratio, trim={4.5cm 1cm 6cm 1cm},clip]{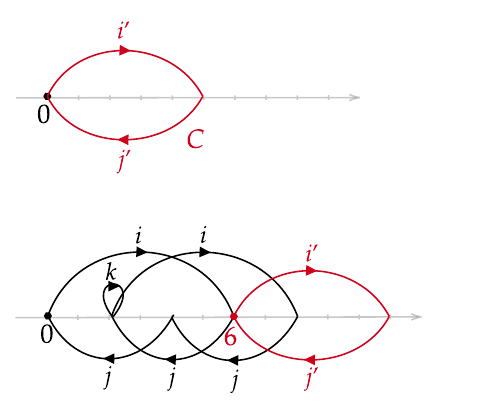}
        \caption{Attaching the circuit $C$ to $G(w)$ at vertex $6$.}
        \label{fig:attach}
    \end{minipage}
\end{figure}

By reading the letters in $w$ one by one and tracing the corresponding edges of $G(w)$, we obtain an Euler path of $G(w)$.
Furthermore, if the word $w$ represents the neutral element (or any element of the form $(y, 0)$), then this Euler path is an Euler circuit.

We point out that the element which $w$ represents is uniquely determined by $G(w)$:

\begin{fct}[Product of associated graph]
Let $w$ be a word over the alphabet $\mG$, and let $G = G(w)$ be its associated $\mG$-graph.
For an edge $e \in E(G)$, denote by $\ell(e)$ the label of $e$, denote by $s(e) \in \Z$ the starting vertex of $e$, then $w$ represents the element
\begin{equation}\label{eq:edges}
\left(\sum_{e \in E(G)} X^{s(e)} \cdot y_{\ell(e)}, \sum_{e \in E(G)} b_{\ell(e)}\right).
\end{equation}
\end{fct}

For an arbitrary $\mG$-graph $G$, the element in Expression~\eqref{eq:edges} will be called the \emph{product} of the graph $G$.
It is easy to see that, if $G$ contains an Eulerian path, then by following this path we obtain a word $w$ that represents the product of $G$.

Let $C$ be a another Eulerian $\mG$-graph (seen as a circuit). 
We define the following action of \emph{attaching $C$ to $G$ at vertex $v$}:
For each edge $e$ of $C$, starting at vertex $s(e)$ with label $\ell(e)$, we add an edge $e'$ to $G$, starting at vertex $s(e) + v$ with label $\ell(e)$. See Figure~\ref{fig:attach} for the illustration of an example.
The product of the resulting graph is uniquely determined by $G$, $C$ as well as $v$:

\begin{fct}[Effect of attaching circuit to a graph]\label{fct:attach}
Let $G$ be an arbitrary $\mG$-graph and $C$ be an Eulerian $\mG$-graph.
Denote by $(y_G, b_G)$ the product of $G$ and by $(y_C, 0)$ the product of $C$.
Then attaching $C$ to $G$ at vertex $v$ results in a graph with product $\left(y_G + X^v \cdot y_C, b_G\right)$.
\end{fct}


\subsection{Group Problem implies polynomial equations}\label{subsec:grptopoly}
In this subsection, we prove the ``only if'' part of Proposition~\ref{prop:wrtoeq}.
We will show that if $\langle \mG \rangle$ is a group then there exists a double-full set $S \subset I \times J$ and polynomials $f_{(i,j)}, f_k \in \NXds$ for $(i, j) \in S, k \in K$ satisfying (i)-(iii) of Proposition~\ref{prop:wrtoeq}.

Recall that $d \coloneqq \gcd\left( \{b_a \mid  (y_a, b_a) \in \mG\} \right).$
Define the alphabet $\hG$ of \emph{radical} elements:
\[
\hG \coloneqq \left\{(\hy_a, \hb_a) \;\middle|\; (y_a, b_a) \in \mG \right\},
\]
where
\begin{align*}
    (\hy_a, \hb_a) \coloneqq
    \begin{cases}
    \left(\frac{y_a}{1 + X^d + \cdots + X^{b_a - d}}, d\right) \quad & a \in I \quad \text{ (or equivalently, $b_a > 0$)}, \\
    \left(\frac{y_a}{1 + X^{-d} + \cdots + X^{- (|b_a| - d)}}, -d\right) \quad & a \in J \quad \text{ (or equivalently, $b_a < 0$)}, \\
    (y_a, 0) \quad & a \in K \quad \text{ (or equivalently, $b_a = 0$)}.
    \end{cases}
\end{align*}
Note that these elements are in $\Q(X) \rtimes \Z$ instead of $\Z \wr \Z$.
Direct computation shows that
\begin{equation}\label{eq:lettorad}
(\hy_a, \hb_a)^{|b_a|/d} = (y_a, b_a)
\end{equation}
for $a \in I \cup J$.
Equation~\eqref{eq:lettorad} can also be taken as the definition of $(\hy_a, \hb_a)$.

Since $\langle \mG \rangle$ is a group, by Lemma~\ref{lem:grpword} there exists a full-image word $w \in \mG^*$ that represents the neutral element.
Replacing the letters $(y_a, b_a)$ in $w$ by the words 
$
\underbrace{(\hy_a, \hb_a) \cdots (\hy_a, \hb_a)}_{|b_a|/d \text{ times}}
$
for every $a \in I \cup J$, we obtain a word $\hw \in \hG^*$.
By Equation~\eqref{eq:lettorad}, $\hw$ also represents the neutral element.
To the word $\hw$ we associate a $\hG$-graph $G(\hw)$.
See Figure~\ref{fig:decomp} for the illustration of an example of $G(\hw)$; one can compare it with Figure~\ref{fig:Gw} which illustrates $G(w)$ for the same $w$.



\begin{figure}[ht]
    \centering
    \begin{minipage}[t]{.45\textwidth}
        \centering
        \includegraphics[width=1.0\textwidth,height=1.0\textheight,keepaspectratio, trim={6.8cm 3.3cm 6.3cm 3.0cm},clip]{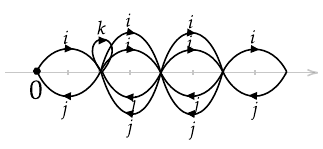}
        \caption{Decomposition of $G(\hw)$. Here $\mG$ and $w$ are the same as in Figure~\ref{fig:Gw}.}
        \label{fig:decomp}
    \end{minipage}
    \hfill
    \begin{minipage}[t]{0.45\textwidth}
        \centering
        \includegraphics[width=1.0\textwidth,height=1.0\textheight,keepaspectratio, trim={6cm 3.3cm 6.2cm 3.1cm},clip]{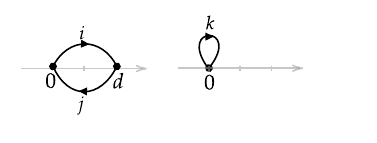}
        \caption{Primitive circuits of type $(i,j)$ and $k$.}
        \label{fig:primitive}
    \end{minipage}
\end{figure}

Since $\hw$ represents the neutral element, the graph $G(\hw)$ is Eulerian.
Since $\hb_a = \pm d$ or $0$ for all indices $a \in I \cup J \cup K$, the graph $G(\hw)$ can be decomposed into two classes of smaller circuits.
The first class of circuits is an edge with some label $i \in I$ (edge directed to the right) followed by an edge with some label $j \in J$ (edge directed to the left); we call such a circuit \emph{of the type} $(i, j)$.
The second class is a loop with label $k \in K$; we call such a circuit \emph{of the type} $k$.
We call these two classes of circuits \emph{primitive}.
See Figure~\ref{fig:primitive} for an illustration.


\begin{restatable}{lem}{lemdecomp}\label{lem:decomp}
    The graph $G(\hw)$ can be constructed by starting with an edgeless graph with a single vertex $0$ and gradually attaching primitive circuits.
\end{restatable}

Recall the definition of $h_{(i,j)}$ in Equation~\eqref{eq:defhij}.
In fact, $h_{(i,j)}$ is the product of a primitive circuit of type $(i,j)$, meaning
$
(\hy_i, \hb_i) \cdot (\hy_j, \hb_j) = (h_{(i,j)}, 0).
$
Similarly, the product of a circuit of type $k$ is $(h_k, 0) \coloneqq (y_k, 0)$.
By Lemma~\ref{lem:decomp}, $G(\hw)$ can be decomposed into primitive circuits.
For each primitive circuit $C$ in the decomposition, denote by $s(C) \in d \Z$ the vertex where $C$ is attached, and by $type(C)$ the type of $C$.
Denote by $\mC$ the set of circuits in the decomposition of $G(\hw)$ into primitive circuits.
By Fact~\ref{fct:attach}, the product of $G(\hw)$ can be written as
\begin{equation}\label{eq:circuits}
    \left(\sum_{C \in \mC} X^{s(C)} \cdot h_{type(C)}, 0\right) = (0, 0).
\end{equation}

For each $(i, j) \in I \times J$ and $k \in K$, define the following  polynomials in $\N[X^{\pm d}]$:
\begin{equation}\label{eq:defonlyiff}
f_{(i,j)} \coloneqq \sum_{\substack{C \in \mC \text{ of type } (i,j)}} \left(X^{d}\right)^{\frac{s(C)}{d}},
\quad 
f_{k} \coloneqq \sum_{\substack{C \in \mC \text{ of type } k}} \left(X^{d}\right)^{\frac{s(C)}{d}}.
\end{equation}
Let 
$
S \coloneqq \{(i, j) \in I \times J \mid f_{(i,j)} \neq 0\}
$.
We point out that $f_k \neq 0$ for all $k \in K$, because $\hw$ is full-image, meaning $G(\hw)$ contains a loop of label $k$ for each $k \in K$.
Equation~\eqref{eq:circuits} becomes
\begin{equation}\label{eq:onlyif}
    \sum_{(i, j) \in S} f_{(i,j)} \cdot h_{(i,j)} + \sum_{k \in K} f_k \cdot y_{k} = 0.
\end{equation}
This is exactly Condition~(i) of Proposition~\ref{prop:wrtoeq}.
It suffices to show the following to complete the proof of the first implication of Proposition~\ref{prop:wrtoeq}.

\begin{restatable}{lem}{lemonlyif}\label{lem:onlyif}
Let $f_{(i,j)}, f_k \in \N[X^{\pm d}]$ and $S \subset I \times J$ be defined as above, then:
\begin{enumerate}[(i)]
    \item $S$ is double-full.
    \item $\deg_{+}\left(\sum_{(i, j) \in S} f_{(i,j)}\right) + d \geq \deg_{+}\left(\sum_{k \in K} f_{k}\right)$.
    \item $\deg_{-}\left(\sum_{(i, j) \in S} f_{(i,j)}\right) \leq \deg_{-}\left(\sum_{k \in K} f_{k}\right)$.
\end{enumerate}
\end{restatable}
\begin{proof}[Sketch of proof]
    For (i), $S$ is double full since $G(\hw)$ contains edges of each type of labels.
    For (ii) and (iii), it suffices to notice that $G(\hw)$ must be connected while $\max(V(G(\hw))) = \deg_{+}\big(\sum_{(i, j) \in S} f_{(i,j)}\big) + d$ and $\min(V(G(\hw))) = \deg_{-}\big(\sum_{(i, j) \in S} f_{(i,j)}\big)$.
\end{proof} 

\begin{proof}[Proof of ``only if'' part of Proposition~\ref{prop:wrtoeq}]
    If $\langle \mG \rangle$ is a group, then there exists a full-image word $w \in \mG^*$ that represents the neutral element.
    Consider $G(\hw)$ and let $f_{(i,j)}, f_k \in \NXds, (i,j) \in S, k \in K$, be as defined in \eqref{eq:defonlyiff}.
    The Conditions~(i)-(iii) of Proposition~\ref{prop:wrtoeq} follow directly from Equation~\eqref{eq:onlyif} and Lemma~\ref{lem:onlyif}.
\end{proof}

\subsection{Polynomial equation implies Group Problem}\label{subsec:polytogrp}
In this subsection, we prove the ``if'' part of Proposition~\ref{prop:wrtoeq}.
Given a double-full set $S \subset I \times J$ and positive polynomials $f_{(i,j)}, f_k \in \NXds$ for $(i, j) \in S, k \in K$ that satisfy Conditions (i)-(iii) of Proposition~\ref{prop:wrtoeq}, we will construct an Eulerian $\mG$-graph $G$ with product zero.
The main difficulty here is that the length of the edges of a $\mG$-graph are no longer identical, as opposed to $\hG$-graphs.
Therefore one can no longer decompose $\mG$-graphs into primitive circuits.
The key idea is a work-around that simulates primitive circuits using longer circuits.

For $(i, j) \in I \times J$, we define an \emph{elementary circuit of the type} $(i, j)$ to be a circuit that starts with $|b_j|$ edges of label $i$, followed by $b_i$ edges of label $j$.
See Figure~\ref{fig:elem} for an example.

\begin{figure}[ht]
    \centering
    \includegraphics[width=0.6\textwidth,height=0.5\textheight,keepaspectratio, trim={4.5cm 3cm 4cm 3cm},clip]{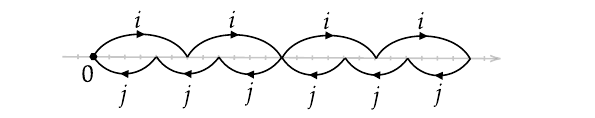}
    \caption{An elementary circuit of type $(i,j)$. Here, $b_i = 6, b_j = -4$.}
    \label{fig:elem}
\end{figure}
\begin{restatable}{lem}{lemA}\label{lem:A}
    Suppose $S \subset I \times J$ be double-full.
    There exists an Eulerian $\mG$-graph $A$ with $d \in V(A)$, obtained by attaching together elementary circuits of types in $S$.
\end{restatable}
\begin{figure}[ht]
    \centering
    \includegraphics[width=0.9\textwidth,height=1.0\textheight,keepaspectratio, trim={3.2cm 2.2cm 3.8cm 0.6cm},clip]{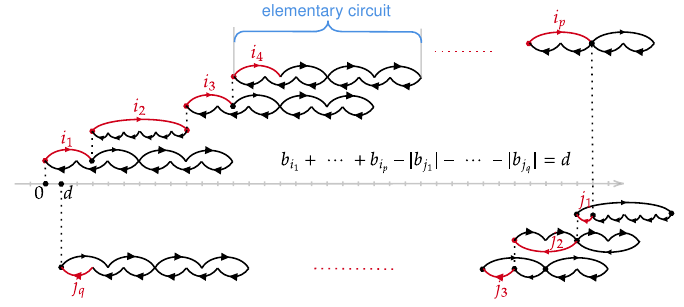}
        \caption{Graph $A$ from Lemma~\ref{lem:A}.}
        \label{fig:A}
\end{figure}

The idea of constructing $A$ is illustrated in Figure~\ref{fig:A}, with a detailed proof given in Appendix~\ref{app:proofs}.
We now characterize the product of $A$.
An elementary circuit of type $(i, j)$ attached at vertex $dv$ contributes $X^{dv} (1 + X^d + \cdots + X^{b_i |b_j| - d}) \cdot h_{(i,j)}$ to the product (see Equation~\eqref{eq:effectele}).
Since $A$ is a combination of elementary circuits, the product of $A$ can written as
$
\sum_{(i, j) \in S} a_{(i, j)} h_{(i,j)}
$
for some $a_{(i, j)} \in \N[X^{\pm d}]$.

Let $N \coloneqq \prod_{i \in I \cup J} |b_i|$.
Note that simultaneously multiplying all $f_{(i,j)}$ and $f_k$ by any polynomial $g \in \N[X^{\pm d}]^*$ does not change the fact that $f_{(i,j)}, f_k \in \N[X^{\pm d}]^*$, and they still satisfy Conditions~(i)-(iii) of Proposition~\ref{prop:wrtoeq}.
Also note that $1 + X^d + \cdots + X^{b_i |b_j| - d} \mid 1 + X^{d} + \cdots + X^{N - d}$.
Therefore, by simultaneously multiplying all $f_{(i,j)}$ and $f_k$ by $p \cdot (X^{-d} + 1 + \cdots + X^{N - 2d})^q$ for large enough $p, q \in \N$, we can suppose that for all $(i, j) \in S$,
\begin{equation}\label{eq:gij}
g_{(i,j)} \coloneqq \frac{f_{(i, j)} - a_{(i, j)} \cdot (1 + X^d + \cdots + X^{N - d})}{1 + X^d + \cdots + X^{b_i |b_j| - d}} \in \N[X^{\pm d}]^* \quad \text{is gap-free,}
\end{equation}
and
\begin{equation}\label{eq:degp}
\deg_{+} \left(f_{(i, j)}\right) > \deg_{+} (a_{(i, j)}) + N - d, \quad\quad \deg_{-} \left(f_{(i, j)}\right) < \deg_{-}(a_{(i, j)}).
\end{equation}

\begin{prop}\label{prop:if}
Suppose $S$ is double-full and $f_{(i,j)}, f_{k} \in \NXds, (i, j) \in S, k \in K$, satisfy Equations~\eqref{eq:gij}, \eqref{eq:degp} and Conditions~(ii)-(iii) of Proposition~\ref{prop:wrtoeq}, then there exists an Eulerian $\mG$-graph $G$ whose product is $\sum_{(i, j) \in S} f_{(i,j)} \cdot h_{(i,j)} + \sum_{k \in K} f_k \cdot y_{k}$.
\end{prop}
\begin{proof}
    We construct $G$ in three steps. See Figure~\ref{fig:final} for an illustration.
    
    \textbf{Step 1: Constructing the foundation $A'$.}
    We start with the $\mG$-graph $A$ constructed in Lemma~\ref{lem:A}.
    We then attach to it another $N/d - 1$ copies of $A$, where the $k$-th copy is attached at vertex $dk$.
    The resulting graph is still Eulerian because $d \in V(A)$.
    We denote by $A'$ the $\mG$-graph obtained by this attachment.
    Then we have $0, d, \cdots, N-d \in V(A')$.
    
    \textbf{Step 2: Attaching elementary circuits of type $(i, j) \in S$.}
    For each pair $(i,j) \in S$, we want to attach elementary circuits of type $(i,j)$ to the graph $A'$, such that the total contribution of these circuits to the product is $g_{(i,j)} \cdot \left(1 + X^d + \cdots + X^{b_i |b_j| - d}\right) \cdot h_{(i, j)}$.
    Write $g_{(i,j)} = \sum_{t = p}^q \gamma_t X^{dt}$.
    We attach a total of $\sum_{t = p}^q \gamma_t$ elementary circuits of type $(i, j)$ to $A'$, where for $t = p, p+1, \ldots, q$, exactly $\gamma_t$ of these circuits are attached at the vertex $dt$.
    The resulting graph is connected (and Eulerian) because $g_{(i,j)}$ is gap-free.
    In fact, for each $t \in [\deg_-(g_{(i,j)})/d, \left(\deg_+(g_{(i,j)}) + b_i |b_j|\right)/d] \cap \Z$ and $u \in [0, N)$, such that $dt \equiv du \mod b_i |b_j|$, the vertex $dt$ is connected to $du \in V(A')$ by a chain of circuits of type $(i,j)$.
    Denote by $A''$ the resulting graph after doing the above attachments for all $(i,j) \in S$.
    Then
    \begin{equation}\label{eq:conn}
    dt \in V(A'') \; \text{ for all } \; t \in [\deg_-(g_{(i,j)})/d, \left(\deg_+(g_{(i,j)}) + b_i |b_j|\right)/d] \cap \Z, \; (i,j) \in S.
    \end{equation}
    
    \textbf{Step 3: Attaching loops of type $k \in K$.}
    For each $k \in K$, we want to attach loops of label $k \in K$ to $A''$, such that the total contribution of these loops to the product is $f_{k} \cdot y_{k}$.
    Write $f_k = \sum_{t = p}^q \beta_t X^{dt}$, we attach a total of $\sum_{t = p}^q \beta_t$ loops of label $k$ to $A''$, where for $t = p, p+1, \ldots, q$, exactly $\beta_t$ of these loops are attached at the vertex $dt$.

    We need to prove that the resulting graph $G$ is still connected (and hence Eulerian).
    In view of Property~\eqref{eq:conn} of the graph $A''$, it suffices to prove $\deg_-(f_k) \geq \min_{(i,j) \in S}\left\{\deg_-(g_{(i,j)})\right\}$ and $\deg_+(f_k) \leq \max_{(i,j) \in S}\left\{\deg_+(g_{(i,j)}) + b_i |b_j|\right\}$ for all $k \in K$.
    By Equations~\eqref{eq:gij} and \eqref{eq:degp}, we have
    $
    \deg_-(g_{(i,j)}) = \deg_-(f_{(i,j)})
    $
    and
    $
    \deg_+(g_{(i,j)})+ b_i |b_j| - d = \deg_+(f_{(i,j)})
    $
    for all $(i,j) \in S$.
    Then, by Conditions~(ii) and (iii) of Proposition~\ref{prop:wrtoeq}, we have 
    \begin{align*}
    & \deg_-(f_k) \geq \deg_-\left(\sum\nolimits_{k \in K} f_k\right) \geq \min_{(i,j) \in S}\{\deg_-(f_{(i,j)})\} = \min_{(i,j) \in S}\{\deg_-(g_{(i,j)})\}, \\
    & \deg_+(f_k) \leq \deg_+\left(\sum\nolimits_{k \in K} f_k\right) \leq \max_{(i,j) \in S}\{\deg_+(f_{(i,j)})\} + d = \max_{(i,j) \in S}\{\deg_-(g_{(i,j)}) + b_i |b_j|\},
    \end{align*}
    for all $k \in K$. Therefore, the resulting graph $G$ is still connected.

    Finally, we count the product of $G$.
    The product of $A'$ is $(1 + X^d + \cdots + X^{N - d}) \cdot \sum_{(i, j) \in S} a_{(i, j)} h_{(i, j)}$.
    The total contribution of elementary circuits in Step 2 is $\sum_{(i, j) \in S} g_{(i,j)} \cdot \left(1 + X^d + \cdots + X^{b_i |b_j| - d}\right) \cdot h_{(i, j)}$.
    The total contribution of loops in Step 3 is $\sum_{k \in K} f_{k} \cdot y_{k}$.
    Thus, by Equation~\eqref{eq:gij}, the product of $G$ is $\sum_{(i, j) \in S} f_{(i,j)} \cdot h_{(i,j)} + \sum_{k \in K} f_k \cdot y_{k}$.
\end{proof}

\begin{figure}[ht]
    \includegraphics[width=1.0\textwidth,height=1.0\textheight,keepaspectratio, trim={1.5cm 0.8cm 4.2cm 0.1cm},clip]{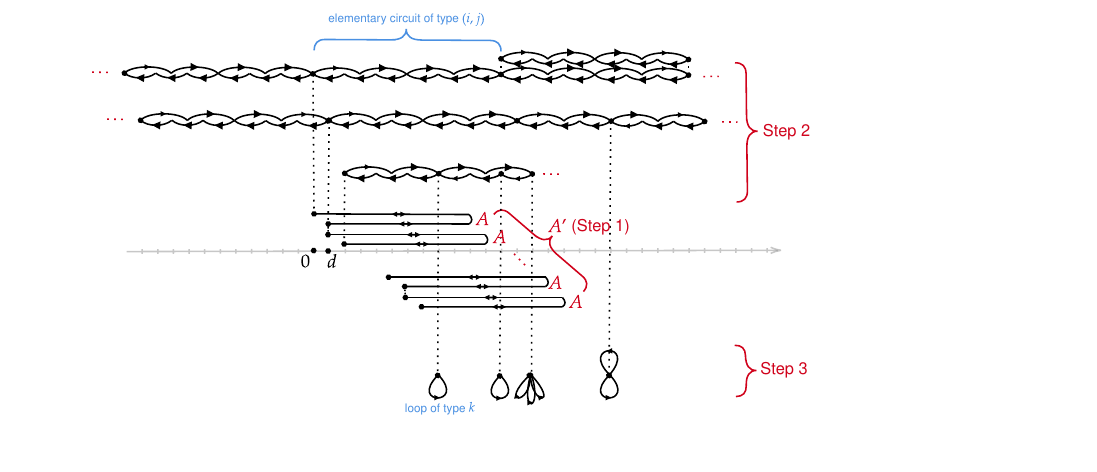}
        \caption{Graph $G$ from Proposition~\ref{prop:if}.}
        \label{fig:final}
\end{figure}

\begin{proof}[Proof of ``if'' part of Proposition~\ref{prop:wrtoeq}]
    We use Proposition~\ref{prop:if} to prove the ``if'' part of Proposition~\ref{prop:wrtoeq}.
    Suppose there exist a double-full set $S \subset I \times J$ and polynomials $f_{(i,j)}, f_k \in \NXds$ for $(i, j) \in S, k \in K$ that satisfy Conditions~(i)-(iii).
    Recall that by simultaneously multiplying all $f_{(i,j)}$ and $f_k$ by $p \cdot (X^{-d} + 1 + X^{d} + \cdots + X^{N - 2d})^q$ for large enough $p, q \in \N$, we can suppose Equations~\eqref{eq:gij} and \eqref{eq:degp} to be satisfied.
    Then Proposition~\ref{prop:if} gives an Eulerian $\mG$-graph $G$ whose product is $\left(\sum_{(i, j) \in S} f_{(i,j)} \cdot h_{(i,j)} + \sum_{k \in K} f_k \cdot y_{k}, 0\right)$.
    This product is equal to the neutral element due to Conditions~(i) of Proposition~\ref{prop:wrtoeq}.
    By following an Eulerian cycle of $G$, we obtain a word $w$ representing the neutral element. The word $w$ is full-image because $G$ contains elementary circuits of all types $(i,j) \in S$ and loops of all types $k \in K$, and because $S$ is double-full.
    Therefore $\langle \mG \rangle$ is a group by Lemma~\ref{lem:grpword}.
\end{proof}

\section{A local-global principle for polynomial equations}\label{sec:locglob}
In this section we prove Proposition~\ref{prop:locglob}.
We follow the line of proof for the original result of Einsiedler et al.~\cite{einsiedler2003does}, while introducing new elements concerning the degree constraints.
See Figure~\ref{fig:locglob} in Appendix~\ref{app:proofs} for an illustration of the proof.

\begin{restatable}{lem}{lemfend}\label{lem:fend}
    Suppose Conditions~(ii) and (iii) of Proposition~\ref{prop:locglob} hold.
    Then there exists $\bff_{end} = (f_{end, S}, f_{end, K}, \cdots) \in \mM$ such that 
    \begin{align}
        & \lc_+(\bff_{end}) \in \R_{>0}^{n}, \quad \deg_{+}\left(f_{end,S}\right) + 1 \geq \deg_{+}\left(f_{end,K}\right), \text{ and} \label{eq:endloc+} \\
        & \lc_-(\bff_{end}) \in \R_{>0}^{n}, \quad \deg_{-}\left(f_{end, S}\right) \leq \deg_{-}\left(f_{end, K}\right). \label{eq:endloc-}
    \end{align}
\end{restatable}

Since $\lc_{\pm}(\bff_{end}) \in \R_{>0}^{n}$, there exists $c > 1$, such that $\bff_{end}(x) \in \R_{> 0}^{n}$ for all $x \in \R_{>0} \setminus [1/c, c]$.
Define the compact set $C \coloneqq [1/4c, 4c]$.

\begin{restatable}{lem}{lemfC}\label{lem:fC}
    Suppose Condition~(i) of Proposition~\ref{prop:locglob} hold.
    Let $C \subset \R_{>0}$ be a compact set, then there exists $\bff_{C} = (f_{C, S}, f_{C, K}, \cdots) \in \mM$ such that 
    $
        \bff_{C}(x) \in \R_{> 0}^{n}
    $
    for all $x \in C$.
\end{restatable}

The key ingredient for finding an element in $\left(\A^+\right)^n$ is the following corollary of Handelman's Theorem.

\begin{restatable}[{Corollary of Handelman's Theorem~\cite{de2001handelman, handelman1985positive}}]{thrm}{corHandelman}\label{cor:Handelman}
    Let $\bff \in \A^{n}$.
    There exists $g \in \A^+$ such that $g \cdot \bff \in \left(\A^+\right)^{n}$ if and only if the two following conditions are satisfied:
    \begin{enumerate}[(i)]
        \item For all $r \in \Rpp$, we have $\bff(r) \in \Rpp^{n}$.
        \item We have $\lc_+(\bff) \in \Rpp^{n}$ and $\lc_-(\bff) \in \Rpp^{n}$.
    \end{enumerate}
\end{restatable}

We now sketch a proof of Proposition~\ref{prop:locglob} by ``gluing'' the two elements $\bff_{end}$ and $\bff_C$ obtained respectively in Lemma~\ref{lem:fend} and \ref{lem:fC}, then applying Theorem~\ref{cor:Handelman}.

\begin{proof}[Sketch of proof of Proposition~\ref{prop:locglob}]
    For an illustration of the proof, see Figure~\ref{fig:locglob}.
    If $\mM$ contains an element $\bff = (f_S, f_K, \cdots) \in \left(\A^+\right)^{n}$ with $\deg_+(f_S) + 1 \geq \deg_+(f_K)$ and $\deg_-(f_S) \leq \deg_-(f_K)$, then simply take $\bff_r = \bff_{\infty} = \bff_{0} = \bff$ for all $r \in \Rpp$:
    Equation~\eqref{eq:locr} is satisfied for all $r$ as well as \eqref{eq:loc+} and \eqref{eq:loc-}.

    Consider the non-trivial direction of implication.
    Let $\bff_{end}, \bff_C \in \mM$ be the elements obtained respectively in Lemma~\ref{lem:fend} and \ref{lem:fC}.
    Define the  polynomial
    $
    q \coloneqq \frac{1}{2c}(X + X^{-1})
    $.
    Let $\epsilon > 0$ be such that 
    $
        \epsilon \cdot \bff_{end}(x) + \bff_{C}(x) \in \Rpp^{n}
    $
    for all $x \in C$.
    Such an $\epsilon$ exists by the compactness of $C$.
    We claim that there exists $N \in \N$ such that $\bff \coloneqq \epsilon q^N \cdot \bff_{end} + \bff_{C}$ satisfies Conditions~(i) and (ii) in Theorem~\ref{cor:Handelman} simultaneously.

    Let $M \in \N$ be such that $\deg_+(f_{end, i}) + M > \deg_+(f_{C, i})$ and $\deg_-(f_{end, i}) - M < \deg_-(f_{C, i})$ for every coordinate $i = S, K, \cdots, n$.
    Let $\bg \coloneqq \epsilon q^M \cdot \bff_{end} + \bff_{C}$.
    Then we have $\lc_+(\bg) = \lc_+(\bff_{end}) \in \Rpp^{n}$ and $\lc_-(\bg) = \lc_-(\bff_{end}) \in \Rpp^{n}$, as well as
    \begin{align*}
    &\deg_{+}\left(g_{S}\right) + 1 = \deg_{+}\left(f_{end,S}\right) + M + 1 \geq \deg_{+}\left(f_{end,K}\right) + M = \deg_{+}\left(g_{K}\right), \\
    &\deg_{-}\left(g_{S}\right) = \deg_{-}\left(f_{end, S}\right) - M \leq \deg_{-}\left(f_{end, K}\right) - M = \deg_{-}\left(g_{K}\right).
    \end{align*}    
    Therefore, there exists a compact set $[1/d, d] \supset C$ such that $\bg(x) \in \Rpp^{n}$ for all $x \in \Rpp \setminus [1/d, d]$.
    Since $[1/d, d]$ is compact, there exists $N > M$ such that
    $
    \epsilon f_{end, i}(x) \cdot 2^{N} + f_{C, i}(x) > 0
    $
    for all $i = S, K, \cdots, n$, and all $x \in [1/d, d]$.
    
    For this $N$, the vector $\bff \coloneqq \epsilon q^N \cdot \bff_{end} + \bff_{C}$ satisfies both Conditions~(i) and (ii) in Theorem~\ref{cor:Handelman}
    (see the full proof in Appendix~\ref{app:proofs}).
    Therefore, we can find $g \in \A^+$ such that $g \bff \in \left(\A^+\right)^{n}$.
    We have at the same time $g \bff \in \mM$ as well as $\deg_+(g f_S) + 1 \geq \deg_+(g f_K)$ and $\deg_-(g f_S) \leq \deg_-(g f_K)$.
    We have thus found the required element $g \bff$.
\end{proof}

\section{Decidability of local conditions}\label{sec:locdec}

In this section we prove Proposition~\ref{prop:declcd} and \ref{prop:declcinf}.
Let $\mM$ be an $\A$-submodule of $\A^{n}$.

\begin{restatable}{lem}{lemposd}\label{lem:posd}
    Let $\bg_1, \ldots, \bg_m$ be a basis of $\mM$ and $r \in \Rpp$.
    There exists $\bff \in \mM$ with $\bff(r) \in \Rpp^{n}$ if and only if there exist $r_1, \ldots, r_m \in \R$ such that $r_1 \bg_1(r) + \cdots + r_m \bg_m(r) \in \Rpp^{n}$.
\end{restatable}

\propdeclcd*
\begin{proof}
    By Lemma~\ref{lem:posd}, the statement to be decided is equivalent to the following sentence in the first order theory of the reals:
    \begin{equation}\label{eq:real}
        \forall r, r>0 \implies \left(\exists r_1 \exists r_2 \cdots \exists r_m, r_1 \bg_1(r) + \cdots + r_m \bg_m(r) \in \Rpp^{n} \right).
    \end{equation}
    Its truth is decidable by Tarski's Theorem~\cite{Tarski1949}.
\end{proof}

For Proposition~\ref{prop:declcinf}, we start by the following definitions.

\begin{defn}
    \begin{enumerate}[(i)]
        \item Suppose $f = \sum_{i = p}^q a_i X^i \in \A \setminus \{0\}$, where $a_p a_q \neq 0$. Define $\init_+(f) \coloneqq a_q X^q$ and $\init_-(f) \coloneqq a_p X^p$. Additionally define $\init_{+} (0) = \init_{-} (0) = 0$.
        \item
        Given $* \in \{+, -\}$, $\balpha = (\alpha_S, \alpha_K, \ldots) \in \Z^{n}$ and $\bff = (f_S, f_K, \ldots) \in \A^{n}$.
        Define $\sgn(*) = 1$ when $* = +$, and $\sgn(*) = -1$ when $* = -$.
        Let
        \[
        m_{*, \balpha}(\bff) \coloneqq \max\{\sgn(*) \cdot \deg_*(f_S) + \alpha_S, \sgn(*) \cdot \deg_*(f_K) + \alpha_K, \ldots\}.
        \]      
        Define
        $
        \init_{*, \balpha}(\bff) \coloneqq (g_S, g_K, \ldots)
        $,
        where for $j = S, K, \ldots$,
        \begin{align*}
        g_j \coloneqq
            \begin{cases}
                \init_*(f_j) \quad & \sgn(*) \cdot \deg_*(f_j) + \alpha_j = m_{*, \balpha}(\bff), \\
                0 \quad & \deg_*(f_j) + \alpha_j < m_{*, \balpha}(\bff).
            \end{cases}
        \end{align*}      
    \end{enumerate}
\end{defn}

For a monomial $f = a_i X^i \in \A$, define $\coef(f) \coloneqq a_i$.
In particular, $\coef(0) = 0$.
Note that for any $* \in \{+, -\}$, $\bff \in \A^{n}$ and $\balpha \in \Z^{n}$, the above defined $\init_{*, \balpha}(\bff)$ is an $n$-tuple of monomials.
Writing $\init_{*, \balpha}(\bff) = (g_S, g_K, \ldots)$, we then extend the definition of $\coef()$ to
\[
\coef(\init_{*, \balpha}(\bff)) \coloneqq (\coef(g_S), \coef(g_K), \ldots) \in \R^{n}.
\]

\begin{restatable}{lem}{lemlctoin}\label{lem:lctoin}
    Fix $* \in \{+, -\}$ and $a \in \Z$.
    The two following conditions are equivalent:
    \begin{enumerate}[(i)]
        \item There exists $\bff = (f_S, f_K, \cdots) \in \mM$ such that 
        \begin{equation}\label{eq:fa}
            \lc_*(\bff) \in \R_{>0}^{n} \quad \text{ and } \quad \deg_{*}\left(f_S\right) + a \geq \deg_{*}\left(f_K\right).
        \end{equation}
        \item There exists $\balpha = (\alpha_S, \alpha_K, \cdots) \in \Z^{n}$ with $\alpha_S - \alpha_K \leq a$, as well as $\bff = (f_S, f_K, \cdots) \in \mM$,
        such that 
        $
        \coef\left(\init_{*, \balpha}(\bff)\right) \in \Rpp^{n}
        $.
    \end{enumerate}
\end{restatable}

We use the notion of a \emph{super Gr\"{o}bner basis} for $\mM$:
see~\cite[Chapter~2]{einsiedler2003does} for its exact definition.
Readers can simply take the following Lemma~\ref{lem:superGB} as its definition, since this will be the only property of the super Gr\"{o}bner basis that we use in this paper.

\begin{lem}[{\cite[Lemma~2.1]{einsiedler2003does}}]\label{lem:superGB}
    Let $* \in \{+, -\}$ and $\balpha \in \Z^{n}$.
    Let $\bg_1, \ldots, \bg_m$ be a super Gr\"{o}bner basis for $\mM$.
    For every $\bg \in \mM$, we have $\init_{*, \balpha}(\bg) = \sum_{i = 1}^m p_i \cdot \init_{*, \balpha}(\bg_i)
    $ for some $p_1, \ldots, p_m \in \A$.
\end{lem}

By~\cite[Chapter~2]{einsiedler2003does}, given a basis for $\mM$, a set of super Gr\"{o}bner basis exists and can be effectively computed.
We now fix a set of a super Gr\"{o}bner basis $\bg_1, \ldots, \bg_m$ for $\mM$.

\begin{restatable}{cor}{corr}\label{cor:r}
    Let $* \in \{+, -\}$ and $\balpha \in \Z^{n}$.
    Then there exists $\bff \in \mM$ with $\coef\left(\init_{*, \balpha}(\bff)\right) \in \Rpp^{n}$ if and only if there exist $r_1, \ldots, r_m \in \R$ with $\sum_{i = 1}^m r_i \cdot \coef(\init_{*, \balpha}(\bg_i)) \in \Rpp^{n}$.
\end{restatable}

\begin{restatable}[{Generalization of~\cite[Lemma~6.1]{einsiedler2003does}}]{lem}{lemfinalpha}\label{lem:finalpha}
    Fix $* \in \{+, -\}$.
    The initial tuples $\init_{*, \balpha}(\bg_1), \ldots, \init_{*, \balpha}(\bg_m)$ can take only a finite number of possible values when $\balpha$ varies in the set
    $
    \{\balpha = (\alpha_S, \alpha_K, \cdots) \in \Z^{n} \mid \alpha_S - \alpha_K \leq a\}
    $.
    Furthermore, one can effectively compute representatives $\balpha_1, \ldots, \balpha_p \in \Z^{n}$, such that the tuples $\init_{*, \balpha_1}(\bg_i), i = 1, \ldots, m, \ldots, \init_{*, \balpha_p}(\bg_i), i = 1, \ldots, m,$ are all the possible tuples when $\alpha$ varies.
\end{restatable}

\propdeclcinf*
\begin{proof}
    First we compute a set of super Gr\"{o}bner basis $\bg_1, \ldots, \bg_m$ for $\mM$.
    Then for each $* \in \{+, -\}$, by Lemma~\ref{lem:finalpha} we compute $\balpha_1, \ldots, \balpha_p \in \Z^{n}$ such that the tuples $\init_{*, \balpha_1}(\bg_i), i = 1, \ldots, m, \ldots, \init_{*, \balpha_p}(\bg_i), i = 1, \ldots, m,$ are all the possible tuples when $\alpha_S - \alpha_K \leq a$.
    For each of these $\balpha \in \{\balpha_1, \ldots, \balpha_p\}$, use linear programming to decide whether there exist real numbers $r_1, \ldots, r_m \in \R$ such that $\sum_{i = 1}^m r_i \cdot \coef\left(\init_{*, \balpha}(\bg_i)\right) \in \Rpp^{n}$.
    By Corollary~\ref{cor:r}, such $r_1, \ldots, r_m \in \R$ exist if and only if there exists $\bff \in \mM$ with $\coef(\init_{*, \balpha}(\bff)) \in \Rpp^{n}$.
    By Lemma~\ref{lem:lctoin}, this is true if and only if there exists $\bff = (f_S, f_K, \cdots) \in \mM$ satisfying condition~\eqref{eq:declcinf}.
\end{proof}

\bibliography{wreath}

\appendix

\section{Omitted proofs}\label{app:proofs}

\lemgrpword*
\begin{proof}
    Let $w \in \mG^*$ be a full-image word with $\pi(w) = I$.
    Then for every $i$, the word $w$ can be written as $w = v g_i v'$, so $g_i^{-1} = \pi(v') \pi(v) \in \langle \mG \rangle$.
    Therefore, the semigroup $\langle \mG \rangle$ contains all the inverse $g_i^{-1}$, and is thus a group.
    
    If $\langle \mG \rangle$ is a group, then for all $i$, the inverse $g_i^{-1}$ can be written as $\pi(w_i)$ for some word $w_i \in \mG^*$.
    Then the word $w \coloneqq g_1 w_1 g_2 w_2 \cdots g_a w_a$ is a full-image word with $\pi(w) = \pi(g_1 w_1) \cdots \pi(g_a w_a) = I$.
\end{proof}

\propeasy*
\begin{proof}
    First suppose $\langle \mG \rangle$ is a group.
    Without loss of generality suppose $I = \emptyset$. Then every element $(y, b) \in \langle \mG \rangle$ must satisfy $b \leq 0$.
    Therefore, for any $(y_j, b_j) \in \mG$ with $b_j < 0$, we have $(y_j, b_j)^{-1} = (X^{-b_j} \cdot y_j, - b_j) \not\in \langle \mG \rangle$ by the positivity of $- b_j$.
    Hence $J = \emptyset$.

    Since $\langle \mG \rangle$ is a group, by Lemma~\ref{lem:grpword} there exists a full-image word $w \in \mG^*$ that represents $(0, 0)$. Let $n_k$ be the number of times the word $(y_k, 0)$ appears in $w$, then $n_k > 0$ and $w$ represents $\left(\sum_{k \in K} n_k y_k, 0\right) = (0, 0)$.
    This finishes the first implication.

    For the converse implication, suppose $I = J = \emptyset$ and $\sum_{k \in K} n_k y_k = 0$ for some integers $n_k \in \Zpp$.
    Then the word $\prod_{k \in K}(y_k, 0)^{n_k} = (0, 0)$.
    Therefore there exists a full-image word that represents $(0, 0)$. Hence $\langle \mG \rangle$ is a group.

    Let $p = \min_{k \in K}\{\deg_{-}(y_k)\}$ and $q = \max_{k \in K}\{\deg_{+}(y_k)\}$.
    For each $k \in K$, write $y_k = \sum_{t = p}^q \beta_{k, t} X^t$, then $\sum_{k \in K} n_k y_k = 0$ is equivalent to the system of linear equations
    \begin{equation}\label{eq:lin}
    \sum_{k \in K} n_k \beta_{k, t} = 0, \quad t = p, p+1, \ldots, q.
    \end{equation}
    Deciding whether the system~\eqref{eq:lin} has solution over $\Zpp$ can be decided using integer programming.
\end{proof}

\corsys*
\begin{proof}
    We show that Corollary~\ref{cor:sys} is equivalent to Proposition~\ref{prop:wrtoeq}.

    By the definition of $h_{(i,j), m}, y_{k, m}, (i,j) \in S, k \in K, m = 0, \ldots, d-1$ in Equation~\eqref{eq:hsys} and \eqref{eq:ysys}, the Equation~\eqref{eq:Xd} in Condition~(i) of Proposition~\ref{prop:wrtoeq} is equivalent to the following system:
    \begin{equation}\label{eq:sys}
        \sum_{(i, j) \in S} f'_{(i,j)} h_{(i,j),m} + \sum_{k \in K} f'_k y_{k, m} = 0, \quad m = 0, \ldots, d-1.
    \end{equation}
    Where $f'_{(i,j)}, f'_k$ are polynomials in $\N[X^{\pm}]^*$ such that $f_{(i,j)} = f'_{(i,j)}(X^d), f_{k} = f'_{k}(X^d)$.

    On one hand, suppose there exist polynomials $f_{(i,j)}, f_k \in \NXds$ that satisfy Conditions~(i)-(iii) of Proposition~\ref{prop:wrtoeq}.
    Then the polynomials $f'_{(i,j)}, f'_k$ satisfying the system~\eqref{eq:sys} are also in $\Rp[X^{\pm d}]$.
    Then let $f'_S = \sum_{(i, j) \in S} f'_{(i,j)}$ and $f'_K = \sum_{k \in K} f'_k$, so Conditions~(i)-(iii) of Corollary~\ref{cor:sys} are satisfied for $f'_S, f'_K, f'_{(i,j)}, f'_k \in \Rpp[X^{\pm}]^*$.

    On the other hand, suppose there exist polynomials $f'_S, f'_K, f'_{(i,j)}, f'_k \in \Rpp[X^{\pm}]^*$ that satisfy Conditions~(i)-(iii) of Corollary~\ref{cor:sys}.
    We show that there exist $f_S, f_K, f_{(i,j)}, f_k \in \N[X^{\pm}]^*, (i, j) \in S, k \in K$ that satisfy the same Equations~\eqref{eq:syscor1}-\eqref{eq:syscor3}, and such that $\deg_{\pm} f'_{(i,j)} = \deg_{\pm} f_{(i,j)}$, $\deg_{\pm} f'_{k} = \deg_{\pm} f_{k}$ for all $(i, j) \in S, k \in K$.

    In fact, by the homogeneity of Equations~\eqref{eq:syscor1}-\eqref{eq:syscor3}, one can multiply all $h_{(i,j), m}$ and $y_{k, m}$ simultaneously by their common denominator, and suppose $h_{(i,j), m}, y_{k, m} \in \Q[X^{\pm}]$.
    Then, \emph{fixing the degrees of $f_{(i,j)}$ and $f_k$}, one can rewrite Equations~\eqref{eq:syscor1}-\eqref{eq:syscor3} as a system of homogeneous linear equation where the variables are the coefficients of $f_{(i,j)}$ and $f_k$.
    We then add to this system of homogeneous linear equations a boolean combination of homogeneous linear inequalities to guarantee $f_{(i,j)}, f_k \in \Rpp[X^{\pm}]^*$ (this can be expressed using inequalities for the coefficients of $f_{(i,j)}, f_k$), as well as to guarantee the degree of $f_{(i,j)}, f_k$.
    Since this system of homogeneous linear equations plus boolean combination of homogeneous linear inequalities has a solution over $\R$, it also has a solution over $\Q$, and even over $\Z$ by the homogeneity.
    Therefore, we obtain a solution over $\Z$ for the coefficients of $f_{(i,j)}, f_k$.
    This gives us a solution $f_{(i,j)}, f_k \in \N[X^{\pm}]^*$ with the same fixed degrees.
    
    Consequently, $f_{(i,j)}, f_k,  (i, j) \in S, k \in K$, satisfy the system~\eqref{eq:sys}.
    Thus, $f_{(i,j)}(X^d), f_k(X^d)$,  $(i, j) \in S, k \in K$. satisfy Conditions~(i)-(iii) of Proposition~\ref{prop:wrtoeq}.
\end{proof}

\lemdecomp*
\begin{proof}
    Denote by $G_0$ the edgeless graph with a single vertex $0$.
    We show that every Eulerian $\hG$-graph $G$ can be constructed by attaching primitive circuits to $G_0$.
    We use induction on the number of edges in $G$.   
    When there is no edge in $G$, it is $G_0$, and we are done.
    
    When there are loops in $G$, these are loops of some label $k \in K$, and are therefore primitive circuits themselves.
    Removing them results in another Eulerian graph $G'$ and decreases the number of edges.
    By the induction hypothesis $G'$ can be constructed by attaching primitive circuits to $G_0$. Then attaching the loops (primitive circuits of type in $K$) to $G'$ results in $G$.

    When there are no loops in $G$, denote $m \coloneqq \max(V(G))$. Then since $G$ is Eulerian, there must be an edge $e$ of label $\ell(e)$ starting from the vertex $m$, and an edge $e'$ of label $\ell(e')$ ending at the vertex $m$. 
    Since all the edges in $G$ are of length $d$, the edge $e$ must end at vertex $m - d$, and $e'$ must start at $m-d$. (The length of an edge of label $a$ is $\hb_a$.)
    Therefore the circuit consisting of $e$ and $e'$ is primitive of type $(\ell(e), \ell(e'))$.
    Removing the circuit (and the vertex $m$ if necessary) results in a graph $G'$.
    $G'$ is connected (and Eulerian) since there is no loop at $m$, and the only possible neighbour of the vertex $m$ is $m-d$.
    By the induction hypothesis $G'$ can be constructed by attaching primitive circuits to $G_0$.
    Then attaching to $G'$ the primitive circuits of type $(\ell(e), \ell(e'))$ at vertex $m-d$ results in $G$. 
\end{proof}

\lemonlyif*
\begin{proof}
(i). Take an arbitrary index $i \in I$, we show that there exists $j \in J$ such that $f_{(i,j)} \neq 0$.
Indeed, the letter $\widehat{(y_i, b_i)}$ appears in $\hw$, so $G(\hw)$ contains an edge with label $i$.
Therefore, in the primary circuit decomposition of $G(\hw)$, there is a circuit of the type $(i, j)$ for some $j \in J$.
Therefore $f_{(i,j)} \neq 0$.
Analogously, for an arbitrary index $j \in J$, we can show that there exists $i \in I$ such that $f_{(i,j)} \neq 0$.
It follows that $S$ is double-full.

(ii). Since $S$ is double-full, it is non-empty.
By construction, the vertex in $V(G(\hw))$ with the largest value is equal to
$\deg_{+}\left(\sum_{(i, j) \in S} f_{(i,j)}\right) + d$.
The ``$+d$'' comes from the fact that a circuit of the type $(i, j)$ starting at vertex $v$ contains a vertex at $v + d$.
By the definition of $f_k$, we have $\deg_{+}\left(\sum_{k \in K} f_{k}\right) \leq \deg_{+}\left(\sum_{(i, j) \in S} f_{(i,j)}\right) + d$.

(iii). The vertex in $V(G(\hw))$ with the smallest value is equal to
$
\deg_{-}\left(\sum_{(i, j) \in S} f_{(i,j)}\right)
$.
The definition of $f_k$ yields $\deg_{-}\left(\sum_{k \in K} f_{k}\right) \geq \deg_{-}\left(\sum_{(i, j) \in S} f_{(i,j)}\right)$.
\end{proof}

\lemA*
\begin{proof}
    Define $d_1 \coloneqq \gcd\{b_i, i \in I\}$ and $d_2 \coloneqq \gcd\{b_j, j \in J\}$, then $d = \gcd(d_1, d_2)$.
    By Bezout's theorem, there exists $n_1, n_2 \in \N$ such that $d_1 n_1 - d_2 n_2 = d$.
    Furthermore, we have $d_1 (n_1 + d_2 n) - d_2 (n_2 + d_1 n) = d$ for all $n \in \N$.
    
    It is well-known that for every set $T \subset \N$ the generated monoid $T^* \coloneqq \{t_1 + \cdots + t_p \mid p \geq 1\}$ is eventually identical with $\gcd(T) \cdot \N$.
    Therefore, for large enough $n \in \N$, we have simultaneously $d_1 (n_1 + d_2 n) \in \{b_i, i \in I\}^*$ and $d_2 (n_2 + d_1 n) \in \{|b_j|, j \in J\}^*$.
    In other words, there exists $i_1, \ldots, i_p \in I$ and $j_1, \ldots, j_q \in J$ such that
    \[
    b_{i_1} + \cdots + b_{i_p} - |b_{j_1}| - \cdots - |b_{j_q}| = d.
    \] 
    
    Recall that $S$ is double-full.
    For each $i \in I$, define $S(i)$ to be an index in $J$ such that $(i, S(i)) \in S$.
    Similarly, for each $j \in J$, define $S(j)$ to be an index in $I$ such that $(S(j), j) \in S$.
    
    We construct $A$ in the following way.
    We start with the edgeless graph with only the vertex $0$.
    We attach to it $p$ elementary circuits, respectively of the type $(i_1, S(i_1)), \ldots, (i_p, S(i_p))$, where the the $k$-th elementary circuit is attached at the vertex $b_{i_1} + \cdots + b_{i_{k-1}}$. (For $k = 1$, the circuit is attached at $0$.)
    It is easy to see that the resulting graph is connected: indeed, an elementary circuit of type $(i_k, S(i_k))$ starting at the vertex $b_{i_1} + \cdots + b_{i_{k-1}}$ will pass through the vertex $b_{i_1} + \cdots + b_{i_{k}}$.

    Next, we attach $q$ more elementary circuits, respectively of the type $(S(j_1), j_1), \ldots, (S(j_q), j_q)$, where the $k$-th elementary circuit is put at the vertex $b_{i_1} + \cdots + b_{i_p} - |b_{j_1}| - \cdots - |b_{j_{k}}|$.
    Again it is easy to see that the resulting graph is connected.
    Furthermore, $d = b_{i_1} + \cdots + b_{i_p} - |b_{j_1}| - \cdots - |b_{j_q}|$ is a vertex of the resulting graph.
    We have therefore constructed the desired graph $A$.
\end{proof}

\lemfend*
\begin{proof}
    Let $M \in \N$ be such that $\deg_-(f_{\infty, i}) + M > \deg_+(f_{0, i})$ for every coordinate $i = S, K, \cdots$.
    Then take $\bff_{end} \coloneqq \bff_0 + X^M \cdot \bff_{\infty}$.
    We have $\lc_+(\bff_{end}) = \lc_+(\bff_{\infty})$ and $\lc_-(\bff_{end}) = \lc_-(\bff_{0})$, as well as 
    \[
    \deg_{+}\left(f_{end,S}\right) + 1 = \deg_{+}\left(f_{\infty,S}\right) + M + 1 \geq \deg_{+}\left(f_{\infty,K}\right) + M = \deg_{+}\left(f_{end,K}\right)
    \]
    and
    \[
    \deg_{-}\left(f_{end, S}\right) = \deg_{-}\left(f_{0, S}\right) \leq \deg_{-}\left(f_{0, K}\right) = \deg_{-}\left(f_{end, K}\right).
    \]
\end{proof}

\lemfC*
\begin{proof}
    By Condition~(i) of Proposition~\ref{prop:locglob}, for each $r \in \R_{> 0}$, there exists $\bff_r \in \mM$ such that $\bff_r(r) \in \R_{>0}^{n}$.
    By the continuity of polynomial functions, there is an open ball $B(r, a_r)$, centered at $r$, with radius $a_r$, such that $\bff_r(x) \in \R_{>0}^{n}$ for all $x \in B(r, a_r)$.
    
    Consider the open cover $B(r, \frac{a_r}{2}), r \in C$ of the set $C$.
    Since $C$ is compact, there is a finite subcover, which we denote by $B(r_1, \frac{a_{r_1}}{2}), \cdots, B(r_m, \frac{a_{r_m}}{2})$.

    For every small enough $\delta > 0$ and for each $1 \leq i \leq m$, there exists a  polynomial $q_i \in \RX$ such that $|q_i (x)| < \delta$ for all $x \in C \setminus B(r_i, a_{r_i})$ and $|q_i (x)| > 1 - \delta$ for all $x \in B(r_i, \frac{a_{r_i}}{2})$.
    Therefore, for small enough $\delta$, the sum $\bff_{C} \coloneqq \sum_{i = 1}^m q_i \cdot \bff_{r_i} \in \mM$ satisfies $\bff_{C}(x) \in \R_{> 0}^{n}$ for all $x \in C$.
\end{proof}

\begin{thrm}[{Handelman's Theorem~\cite{de2001handelman, handelman1985positive}}]\label{lem:Handelman}
    Let $f \in \A$ be a  polynomial.
    There exists $g \in \A^+$ such that $fg \in \A^+$ if and only if the two following conditions are satisfied:
    \begin{enumerate}[(i)]
        \item For all $r \in \Rpp$, we have $f(r) > 0$.
        \item We have $\lc_+(f) > 0$ and $\lc_-(f) > 0$.
    \end{enumerate}
\end{thrm}

\corHandelman*
\begin{proof}
    If there exists $g \in \A^+$ such that $g \cdot \bff \in \left(\A^+\right)^{n}$, then obviously (i) and (ii) are satisfied.

    One the other hand, let $\bff = (f_1, \ldots, f_{n}) \in \A^{n}$ satisfying (i) and (ii).
    By Handelman's theorem (Theorem~\ref{lem:Handelman}), there exist $g_1, \ldots, g_{n} \in \A^{+}$ such that $f_1 g_1 \in \A^+, \ldots, f_{n} g_{n} \in \A^+$.
    Let $g \coloneqq g_1 g_2 \cdots g_{n}$, then $g \cdot \bff \in \left(\A^+\right)^{n}$.
\end{proof}

\proplocglob*
\begin{proof}
    If $\mM$ contains an element $\bff = (f_S, f_K, \cdots) \in \left(\A^+\right)^{n}$ with $\deg_+(f_S) + 1 \geq \deg_+(f_K)$ and $\deg_-(f_S) \leq \deg_-(f_K)$, then simply take $\bff_r = \bff_{\infty} = \bff_{0} = \bff$ for all $r$:
    Equation~\eqref{eq:locr} is satisfied for all $r$ as well as \eqref{eq:loc+} and \eqref{eq:loc-}; hence all three conditions are satisfied.

    Consider the non-trivial direction of implication: let $\bff_{end}, \bff_C \in \mM$ be the elements obtained respectively in Lemma~\ref{lem:fend} and \ref{lem:fC}.
    Define the  polynomial
    \[
    q \coloneqq \frac{1}{2c}(X + X^{-1}) \in \RX.
    \]
    Let $\epsilon > 0$ be such that 
    \begin{equation}\label{eq:ineqepapp}
        \epsilon \cdot \bff_{end}(x) + \bff_{C}(x) \in \Rpp^{n}
    \end{equation}
    for all $x \in C$.
    Such a $\epsilon$ exists by the compactness of $C$.
    We claim that there exists $N \in \N$, the vector $\bff \coloneqq \epsilon q^N \cdot \bff_{end} + \bff_{C}$ satisfies Conditions~(i) and (ii) in Corollary~\ref{cor:Handelman} simultaneously.

    Let $M \in \N$ be such that $\deg_+(f_{end, i}) + M > \deg_+(f_{C, i})$ and $\deg_-(f_{end, i}) - M < \deg_-(f_{C, i})$ for every coordinate $i = S, K, \cdots$.
    Let $\bg \coloneqq \epsilon q^M \cdot \bff_{end} + \bff_{C}$.
    Then we have $\lc_+(\bg) = \lc_+(\bff_{end}) \in \Rpp^{n}$ and $\lc_-(\bg) = \lc_-(\bff_{end}) \in \Rpp^{n}$, as well as
    \[
    \deg_{+}\left(g_{S}\right) + 1 = \deg_{+}\left(f_{end,S}\right) + M + 1 \geq \deg_{+}\left(f_{end,K}\right) + M = \deg_{+}\left(g_{K}\right),
    \]
    \[
    \deg_{-}\left(g_{S}\right) = \deg_{-}\left(f_{end, S}\right) - M \leq \deg_{-}\left(f_{end, K}\right) - M = \deg_{-}\left(g_{K}\right).
    \]    
    Therefore, there exists another compact set $[1/d, d] \supset C$ such that $\bg(x) \in \Rpp^{n}$ for all $x \in \Rpp \setminus [1/d, d]$.
    Since $[1/d, d]$ is compact, there exists $N > M$ such that
    \begin{equation}\label{eq:ineq2Napp}
    \epsilon f_{end, i}(x) \cdot 2^{N} + f_{C, i}(x) > 0
    \end{equation}
    for all $i = S, K, \cdots,$ and all $x \in [1/d, d]$.
    We prove that for this $N$, the vector $\bff \coloneqq \epsilon q^N \cdot \bff_{end} + \bff_{C}$ satisfies Equations Conditions~(i) and (ii) in Corollary~\ref{cor:Handelman}  simultaneously.

    Fix any coordinate $i = S, K, \cdots$.
    For every $x \in \Rpp \setminus [1/d, d]$, we have
    \[
    f_{i}(x) \coloneqq \epsilon q(x)^N \cdot f_{end, i}(x) + f_{C, i}(x) \geq \epsilon q(x)^M \cdot f_{end, i}(x) + f_{C, i}(x) = g_{i}(x) > 0.
    \]
    
    For every $x \in [1/d, d] \setminus C$,
    \[
    f_{i}(x) = \epsilon f_{end, i}(x) \cdot \left(\frac{x + x^{-1}}{2c}\right)^N + f_{C, i}(x) \geq \epsilon f_{end, i}(x) \cdot 2^N + f_{C, i}(x) > 0
    \]
    by Inequality~\eqref{eq:ineq2Napp} and $x + x^{-1} > 4c$.
    
    For every $x \in C \setminus [1/c, c]$,
    \[
    f_{i}(x) = \epsilon q(x)^N \cdot f_{end, i}(x) + f_{C, i}(x) > 0
    \]
    since $f_{end, i}(x) > 0$ and $f_{C, i}(x) > 0$.
    
    For every $x \in [1/c, c]$,
    \[
    f_{i}(x) = \epsilon f_{end, i}(x) \cdot \left(\frac{x + x^{-1}}{2c}\right)^N + f_{C, i}(x) \geq \min\{\epsilon f_{end, i}(x), 0\} + f_{C, i}(x) > 0
    \]
    by Inequality~\eqref{eq:ineqepapp} and $x + x^{-1} < 2c$.
    Therefore, for every $x \in \Rpp$, we have $f_i(x) > 0$.
    In other words, $\bff$ satisfies Conditions~(i) in Corollary~\ref{cor:Handelman} .
    
    Furthermore, since $N > M$ we have $\deg_+(f_{end, i}) + N > \deg_+(f_{C, i})$ and $\deg_-(f_{end, i}) - N < \deg_-(f_{C, i})$ for every coordinate $i = S, K, \cdots$.
    Hence $\lc_+(\bff) = \lc_+(\bff_{end}) \in \Rpp^{n}$ and $\lc_-(\bff) = \lc_-(\bff_{end}) \in \Rpp^{n}$.
    Therefore, $\bff$ satisfies Conditions~(ii) in Corollary~\ref{cor:Handelman}, and
    \[
    \deg_{+}\left(f_{S}\right) + 1 = \deg_{+}\left(f_{end,S}\right) + 1 \geq \deg_{+}\left(f_{end,K}\right) = \deg_{+}\left(f_{K}\right)
    \]
    and
    \[
    \deg_{-}\left(f_{S}\right) = \deg_{-}\left(f_{end, S}\right) \leq \deg_{-}\left(f_{end, K}\right) = \deg_{-}\left(f_{K}\right).
    \]    
    
    Therefore, by Corollary~\ref{cor:Handelman}, we have find $g \in \A^+$ such that $g \bff \in \left(\A^+\right)^{n}$.
    We have at the same time $g \bff \in \mM$ as well as $\deg_+(g f_S) + 1 \geq \deg_+(g f_K)$ and $\deg_-(g f_S) \leq \deg_-(g f_K)$.
    We have thus found the required element $g \bff$.
\end{proof}

\begin{figure}[ht]
    \includegraphics[width=1.0\textwidth,height=1.0\textheight,keepaspectratio, trim={0cm 0cm 0cm 0cm},clip]{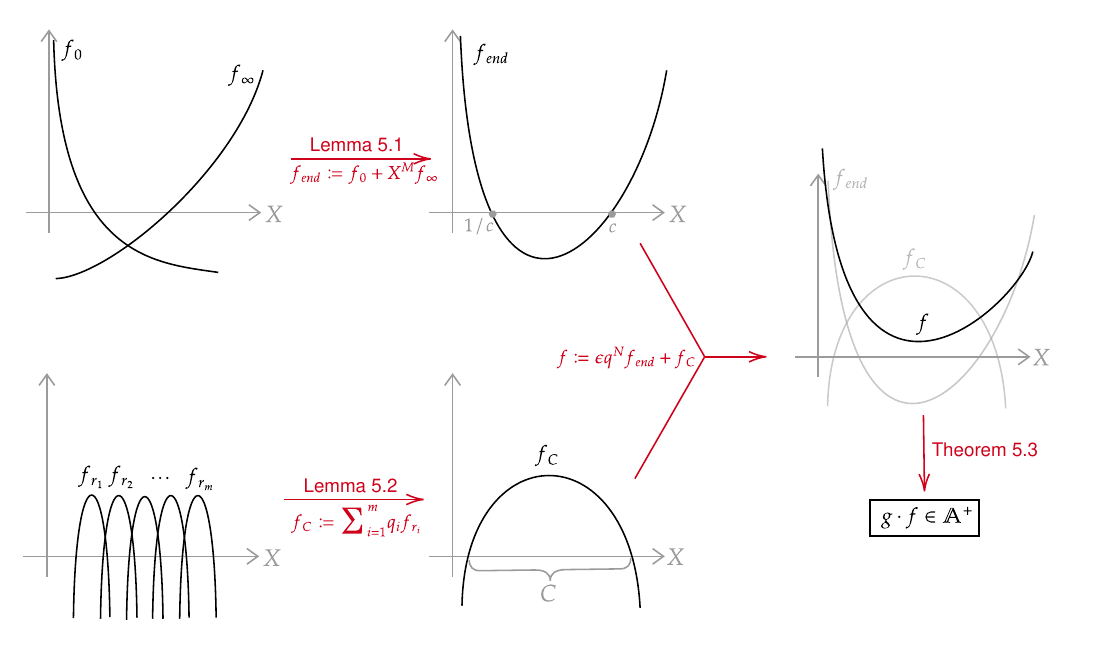}
        \caption{Proof of Proposition~\ref{prop:locglob}, illustrated at dimension $n = 1$.}
        \label{fig:locglob}
\end{figure}

\lemposd*
\begin{proof}
    Suppose there exists $\bff \in \mM$ with $\bff(r) \in \Rpp^{n}$.
    Write $\bff = \phi_1 \bg_1 + \cdots + \phi_m \bg_m$ with $\phi_1, \ldots, \phi_m \in \A$.
    Take $r_i \coloneqq \phi_i(r), i = 1, \ldots, m$, then $r_1 \bg_1(r) + \cdots + r_m \bg_m(r) = \phi_1(r) \bg_1(r) + \cdots + \phi_m(r) \bg_m(r) \in \Rpp^{n}$.

    Suppose there exists real numbers $r_1, \ldots, r_m \in \R$ such that $r_1 \bg_1(r) + \cdots + r_m \bg_m(r) \in \Rpp^{n}$.
    Then taking $\bff \coloneqq r_1 \bg_1 + \cdots + r_m \bg_m \in \mM$, we have $\bff(r) \in \Rpp^{n}$.
\end{proof}

\lemlctoin*
\begin{proof}
    Without loss of generality suppose $*$ is the $+$ sign, otherwise simply change the variable $X$ to $X^{-1}$ and the same argument stays valid.
    
    \textbf{(i) $\implies$ (ii).}
    Suppose $\bff = (f_S, f_K, f_3, f_4, \cdots, f_{n}) \in \mM$ satisfy Condition~\eqref{eq:fa}.
    Let 
    \[
    \balpha = (\alpha_S, \alpha_K, \cdots) \coloneqq \left(- \deg_+(f_S), - \deg_+(f_K), - \deg_+(f_3), \ldots, - \deg_+(f_n) \right).
    \]
    Then $\alpha_S - \alpha_K = - \deg_+(f_S) + \deg_+(f_K) \leq a$, and $m_{*, \balpha}(\bff) = \deg_+(f_S) + \alpha_S = \deg_+(f_K) + \alpha_K = \cdots = \deg_+(f_n) + \alpha_n = 0$, so
    \[
    \init_{+, \balpha}(\bff) = \left(\init_+(f_S), \ldots, \init_+(f_K) \right).
    \]
    Therefore $\coef\left(\init_{+, \balpha}(\bff)\right) = \lc_+(\bff) \in \Rpp^{n}$.

    \textbf{(ii) $\implies$ (i).}
    Suppose there exists $\balpha = (\alpha_S, \alpha_K, \cdots) \in \Z^{n}$ with $\alpha_S - \alpha_K \leq a$, as well as $\bff = (f_S, f_K, f_3, f_4, \cdots, f_n) \in \mM$ that satisfy
    \[
    \coef\left(\init_{*, \balpha}(\bff)\right) \in \Rpp^{n}.
    \]
    In other words, we have
    \[
    \deg_+(f_S) + \alpha_S = \deg_+(f_K) + \alpha_K = \cdots = \deg_+(f_n) + \alpha_n = m_{*, \balpha}(\bff)
    \]
    and
    \[
    \lc_+(\bff) \in \R_{>0}^{n}.
    \]
    The inequality $\deg_{+}\left(f_S\right) + a \geq \deg_{+}\left(f_K\right)$ follows from $\deg_+(f_S) + \alpha_S = \deg_+(f_K) + \alpha_K$ and $\alpha_S - \alpha_K \leq a$.
\end{proof}

\corr*
\begin{proof}
    Let $\bff \in \mM$ with $\coef \left(\init_{*, \balpha}(\bff)\right) \in \Rpp^n$.
    By Lemma~\ref{lem:superGB}, $\init_{*, \balpha}(\bff) = \sum_{i = 1}^m p_i \cdot \init_{*, \balpha}(\bg_i)$ for some $p_1, \ldots, p_m \in \A$.
    Then let $M \coloneqq \{i \mid \deg_*(p_i) + m_{*, \balpha}(\bg_i) = m_{*, \balpha}(\bff)\}$.
    Without loss of generality suppose the initial terms in the sum does not vanish, that is,
    \[
    \sum_{i \in M} \init_*(p_i) \cdot \init_{*, \balpha}(\bg_i) \neq 0.
    \]
    Otherwise we can replace $p_i$ with $p_i - \init_*(p_i)$.
    Since $\sum_{i \in M} \init_*(p_i) \cdot \init_{*, \balpha}(\bg_i)$ does not vanish, we have
    \[
    \init_{*, \balpha}(\bff) = \sum_{i \in M} \init_*(p_i) \cdot \init_{*, \balpha}(\bg_i).
    \]
    Taking respective coefficient of each entry yields 
    \[
    \sum_{i \in M} \lc_*(p_i) \cdot \coef\left(\init_{*, \balpha}(\bg_i) \right) = \coef \left(\init_{*, \balpha}(\bff)\right) \in \Rpp^n.
    \]
    We then take $r_i \coloneqq \lc_*(p_i)$ for $i \in M$ and $r_i \coloneqq 0$ for $i \not\in M$ to conclude for the first implication.

    For the second implication, let $r_1, \ldots, r_m \in \R$ be such that $\sum_{i = 1}^m r_i \cdot \coef(\init_{*, \balpha}(\bg_i)) \in \Rpp^{n}$.
    Take 
    \begin{equation}\label{eq:ff}
    \bff \coloneqq \sum_{i = 1}^m r_i X^{- m_{*, \balpha}(\bg_i)} \cdot \bg_i.
    \end{equation}
    We have, for all $i = 1, \ldots, m$,
    \[
    m_{*, \balpha}\left(r_i X^{- m_{*, \balpha}(\bg_i)} \cdot \bg_{i}\right) = 0.
    \]
    Since $\sum_{i = 1}^m r_i \cdot \coef(\init_{*, \balpha}(\bg_i)) \in \Rpp^{n}$, we have
    \[
    \sum_{i = 1}^m \init_{*, \balpha}\left(r_i X^{- m_{*, \balpha}(\bg_i)} \cdot \bg_i\right) \neq 0.
    \]
    Therefore
    \[
    \sum_{i = 1}^m \init_{*, \balpha}\left(r_i X^{- m_{*, \balpha}(\bg_i)} \cdot \bg_i\right) = \init_{*, \balpha}\left( \sum_{i = 1}^m  r_i X^{- m_{*, \balpha}(\bg_i)} \cdot \bg_i\right) = \init_{*, \balpha}(\bff).
    \]
    Furthermore, at each coordinate $\init_{*, \balpha}(\bff)$ and $\init_{*, \balpha}\left(r_i X^{- m_{*, \balpha}(\bg_i)} \cdot \bg_i\right)$ have the same degree for all $i$.
    Taking the coefficient yields $\coef(\init_{*, \balpha}(\bff)) = \sum_{i = 1}^m r_i \cdot \coef(\init_{*, \balpha}(\bg_i)) \in \Rpp^{n}$.
\end{proof}

\lemfinalpha*
\begin{proof}
    We show that for each $i = 1, \ldots, m$, there is only a finite number of possible initial elements $\init_{*, \balpha}(\bg_i)$ when $\alpha$ varies.
    Without loss of generality, suppose $*$ is the $+$ sign.
    Write $\bg_i = (g_{i, S}, g_{i, K}, \cdots, g_{i, n})$ and $\balpha = (\alpha_S, \alpha_K, \cdots, \alpha_n)$.
    Then $\init_{*, \balpha}(\bg_i)$ depends solely on the order of the integers $\alpha_S + \deg_+(g_{i, S}), \alpha_K + \deg_+(g_{i, K}), \ldots, \alpha_n + \deg_+(g_{i, n})$.
    There is only a finite number of possible orders (including equality).
    For each such order one can compute a polyhedron $\mC \subset \R^{n}$ with rational slopes that contains all $\balpha$ leading to this order.
    Doing this for all $i = 1, \ldots, m$ and taking all possible intersections of these polyhedra yields a finite number of polyhedra.
    For each one of these polyhedra $\mC'$, all $\balpha \in \mC'$ lead to the same $\init_{*, \balpha}(\bg_1), \ldots, \init_{*, \balpha}(\bg_m)$.
    Consider all polyhedra containing an integer point, choosing an integer point $\balpha_{\mC'}$ in each of these polyhedra yields Lemma~\ref{lem:finalpha}.
\end{proof}

\section{Procedure for deciding Group Problem in $\Z \wr \Z$.}\label{app:alg}
In this section of the appendix we give the procedure for deciding Group Problem in $\Z \wr \Z$.
The justification of each step is given in parentheses with reference to the corresponding lemmas or propositions.
\begin{algorithm}[!ht]
\caption{Algorithm for deciding Group Problem in $\Z \wr \Z$}
\label{alg:gp}
\begin{description} 
\item[Input:] 
A finite set $\mG = \{(y_a, b_a) \mid a \in A\}$ in $\Z \wr \Z$.
\item[Output:] \textbf{True} or \textbf{False}.
\end{description}
\begin{enumerate}[Step 1:]
    \item Define the index sets $I, J, K$ as in \eqref{eq:defijk}.
    \item If $I = \emptyset$ or $J = \emptyset$, perform the following.
    (c.f.\ Proposition~\ref{prop:easy})
    \begin{enumerate}
        \item If $I$ and $J$ are not both empty, return \textbf{False}.
        \item If $I = J  = \emptyset$.
        Using integer programming, decide whether $\sum_{k \in K} n_k y_k = 0$ for some $n_k \in \Zpp$. If yes, return \textbf{True}, otherwise return \textbf{False}.
    \end{enumerate}
    \item If $I \neq \emptyset$ and $J \neq \emptyset$, find all double-full sets $S \subset I \times J$. Le $\mS$ be the set of all such $S$.
    \item\label{step:eachS} For each $S \in \mS$, perform the following.
    \begin{enumerate}
        \item Initialize three boolean variables $B_r, B_{\infty}, B_{0}$ as $false$.
        \item For each $(i, j) \in S$, compute $h_{(i,j)}$ as defined in \eqref{eq:defhij}.
        Then, compute $h_{(i,j), m}, y_{k, m}$ for all $(i, j) \in S, k \in K, m = 0, \ldots, d-1$ as defined in \eqref{eq:hsys} and \eqref{eq:ysys}. (c.f.\ Proposition~\ref{prop:wrtoeq}, Corollary~\ref{cor:sys})
        \item Compute a basis $\bg_1, \ldots, \bg_{m'}$ for the module $\mM$ defined in \eqref{eq:defM}.
        \item Decide whether the sentence~\eqref{eq:real} is true using Tarski's theorem. If true, set $B_r \coloneqq true$. (c.f.\ Proposition~\ref{prop:declcd})
        \item Compute a super Gr\"{o}bner basis $\bg_1, \ldots, \bg_m$ for $\mM$.
        \item Let $* = +$, compute $\balpha_1, \ldots, \balpha_p \in \Z^{n}$ defined in Lemma~\ref{lem:finalpha}.
        \item For each of these $\balpha$, decide whether there exist $r_1, \ldots, r_m \in \R$ such that $\sum_{i = 1}^m r_i \cdot \coef\left(\init_{+, \balpha}(\bg_i)\right) \in \Rpp^{n}$. If this is true for some $\balpha$, set $B_{\infty} = true$. (c.f.\ Proposition~\ref{prop:declcinf})
        \item\label{substep:-} Let $* = -$, compute $\balpha_1, \ldots, \balpha_p \in \Z^{n}$ defined in Lemma~\ref{lem:finalpha}.
        \item For each of these $\balpha$, decide whether there exist $r_1, \ldots, r_m \in \R$ such that $\sum_{i = 1}^m r_i \cdot \coef\left(\init_{-, \balpha}(\bg_i)\right) \in \Rpp^{n}$. If this is true for some $\balpha$, set $B_{0} = true$. (c.f.\ Proposition~\ref{prop:declcinf})
        \item If $B_r, B_{\infty}, B_{0}$ are all $true$ at this moment, return \textbf{True}. Otherwise continue to the next $S$. (c.f.\ Proposition~\ref{prop:locglob})
    \end{enumerate}
    \item If the procedure did not return true for any $S$ in Step~\ref{step:eachS}, then return \textbf{False}.
\end{enumerate}
\end{algorithm}

\end{document}